\documentclass[reqno,11pt]{amsart}
\numberwithin{equation}{section}
\usepackage{amsmath}
\usepackage{esint} 
\usepackage{amsthm}
\usepackage{epsfig}
\usepackage{psfrag}
\usepackage{graphicx}
\usepackage{graphpap,latexsym,epsf}
\usepackage{color}
\usepackage{amssymb,mathrsfs,enumerate}
\definecolor{citegreen}{rgb}{0,0.6,0}
\definecolor{refred}{rgb}{0.8,0,0}
\usepackage[colorlinks, citecolor=citegreen, linkcolor=refred]{hyperref}
\newcommand{\R}{\mathbb{R}}

\newcommand{\Sph}{\mathbb{S}}
\def\DDD{{\rm D}}
\def\GGG{{\rm G}}
\def\HHH{{\rm H}}
\def\RRR{{\mathrm R}}
\def\WWW{{\mathrm W}}
\def\SSS{{\rm S}}
\def\EEE{{\rm E}}
\def\TTT{\stackrel{\circ}{\rm R}}

\newcommand{\pa}{\partial}
\newcommand{\na}{\nabla}
\newcommand{\Om}{\Omega}
\newcommand{\De}{\Delta}

\newcommand{\ep}{\varepsilon}
\newcommand{\Rm}{{\rm Rm}}
\newcommand{\Ric}{{\rm Ric}}
\newcommand{\rmd}{{\rm d}}

\mathchardef\emptyset="001F

\definecolor{vgreen}{rgb}{0.1,0.5,0.2}
\definecolor{viola}{RGB}{85,26,139}
\renewcommand{\theequation}{\thesection.\arabic{equation}}

\newtheorem{theorem}{Theorem}[section]
\newtheorem{remark}[theorem]{Remark}
\newtheorem{corollary}[theorem]{Corollary}

\newtheorem{proposition}[theorem]{Proposition}

\newtheorem{lemma}[theorem]{Lemma}

\begin{document}
\title[Riemannian aspects of potential theory]
{Riemannian aspects of potential theory}


\author[V.~Agostiniani]{Virginia Agostiniani}
\address{V.~Agostiniani, Mathematical Institute\\
Woodstock Road\\
Oxford OX2 6GG, UK
} 
\email{Virginia.Agostiniani@maths.ox.ac.uk}

\author[L.~Mazzieri]{Lorenzo Mazzieri}
\address{L.~Mazzieri, Scuola Normale Superiore di Pisa\\
Piazza Cavalieri 7\\
56126 Pisa, Italy
}
\email{l.mazzieri@sns.it}


\begin{abstract} 
In this paper we provide a new method for establishing the rotational symmetry of the solutions to a couple of very classical overdetermined problems arising in potential theory, in both the exterior and the interior punctured domain. 
Thanks to a conformal reformulation of the problems, we obtain Riemannian manifolds with zero Weyl tensor satisfying a quasi-Einstein type equation. Exploiting these geometric properties, we conclude via a splitting argument that the manifolds obtained are half cylinders. 
In turn, the rotational symmetry of the potential is implied.
To the authors' knowledge, some of the overdetermining conditions considered here are new.
\end{abstract}

\maketitle


\smallskip
\noindent 

\date{\today}

\maketitle


\section{Introduction and statement of the results}
\label{s:intro}
In this paper we deal with two classical elliptic boundary value problems from potential theory, namely
\begin{equation}
\label{eq:problemi}
\left\{
\begin{array}{rcll}
\displaystyle
-\Delta u\!\!\!\!&=&\!\!\!\!0 & {\rm in }\quad\R^n\setminus\overline\Omega,\\
\displaystyle
  u\!\!\!\!&=&\!\!\!\!c  &{\rm on }\ \ \pa\Om,\\
\displaystyle
u(x)\!\!\!\!&\to&\!\!\!\!0 & \mbox{as }\ |x|\to+\infty,
\end{array}
\right.
\qquad\mbox{and}\qquad
\left\{
\begin{array}{rcll}
\displaystyle
-\De u\!\!\!\!&=&\!\!\!\!0 & {\rm in }\quad\Om\setminus\{0\},\\
\displaystyle
  u\!\!\!\!&=&\!\!\!\!c  &{\rm on }\ \ \pa\Om,\\
\displaystyle
u(x)\!\!\!\!&\to&\!\!\!\!+\infty & \mbox{as }\ x\to0.
\end{array}
\right.
\end{equation}
Here and throughout this paper, $\Om\subset\R^n$, $n\geq3$, is a bounded domain with smooth boundary which contains the origin.
It is well known that if $\Om$ is the ball $B_{r_0}$ of radius $r_0$ and centred at zero, then the function 
\begin{equation*}
u_0(x)=\frac{|x|^{2-n}}{(n-2)|\Sph^{n-1}|}
\end{equation*}
satisfies both problems with $c=r_0^{2-n}/(n-2)|\Sph^{n-1}|$.
Here, $|\Sph^{n-1}|$ is the area of the $(n-1)$-dimensional unit sphere.
In the following we will revisit for both problems some classical results where the radial symmetry of the solution, and thus of the domain, is derived, provided some extra assumptions are imposed so that the problems become overdetermined. It is worth pointing out that the literature on this subject is vaste so that giving an exhausted list of the contributors is out of reach. Here, we limit ourselves to mention the results that are more related to the subsequent discussion:
\cite{Pay_Phi_1,Ga_Sa,Fra_Ga_2008,Cra_Fra_Ga} for the exterior domain,
\cite{Pay_Sch,Enc_Per} for the interior punctured domain.

In order to accomplish our programme, we propose an unifying point of view, which relies on a suitable reformulation of the problems in the framework of Riemannian geometry. In particular, we are going to exploit the conformal structure of the domains under consideration to obtain an equivalent formulation of the problems which is suitable to be treated via splitting methods. We refer the reader to the recent work \cite{Farina} for other interesting applications of splitting principles in the study of overdetermined problems.
   
To describe our approach in more detail, we set some notation and recall some basic geometric facts.
Throughout this paper, we let $g_{\R^n}$ denote the \emph{Euclidean metric} of $\R^n$. Recall that, in $\R^n\setminus\{0\}$, the metric
$g_{\R^n}$ can be written as
$
dr\otimes dr+r^2g_{\Sph^{n-1}},
$ 
where $g_{\Sph^{n-1}}$ is the metric induced by $g_{\R^n}$ on the unit sphere $\Sph^{n-1}$ and $r>0$ represents the Euclidean distance to the origin.
The \emph{cylindrical metric} on $\R^n\setminus\{0\}$ is instead given by
$
g_{cyl}=r^{-2}dr\otimes dr+g_{\Sph^{n-1}}
$  
and it is related to $g_{\R^n}$ by the conformal factor $r^{-2}$.
Note that, setting $t:=\log r$, one obtains the expression
\[
g_{cyl}=r^{-2}g_{\R^n}=dt\otimes dt+g_{\Sph^{n-1}}.
\]
Geometrically, the variable $t$ can be viewed as the signed cylindrical distance to the hypersurface $\{r=1\}$ or, equivalently, to $\{t=0\}$.
It is clear from this expression that the cylindrical metric is a product metric.
To describe the heuristic of our approach, let us remark that when $\Om$ is the ball $B_{r_0}$ and $u_0$ is the radial solution introduced above, the conformal change which relates the Euclidean metric to the cylindrical one can be written in terms of $u_0$. More precisely,
\[
g_{cyl}=\big((n-2)|\Sph^{n-1}|\,u_0\big)^{\frac2{n-2}}\,g_{\R^n}.
\] 
In general, without assuming a priori that $\Om$ is a ball, we take the solution $u$ to one of our problems and we build the ansatz metric
\[
g=u^{\frac2{n-2}}\,g_{\R^n}
\]
on the manifold with boundary $M$, which eventually corresponds to $\R^n\setminus\Om$ or to $\overline\Om\setminus\{0\}$.
We aim to show that under suitable extra conditions (see, for example, the hypotheses of Theorems \ref{thm:Lore_Virgi} and Theorem  \ref{thm:nabla_const} below) the metric $g$ is the cylindrical metric, up to a multiplicative constant, so that $(M,g)$ is isometric to one half round cylinder. Once this has been established, the rotational symmetry of $u$ follows at once. 

The advantage of this strategy is that one can divide the problem into two steps which can be treated by means of very general principles. The first step consists in proving that $g$ is the Riemannian product of a half line with a level set of $u$. To do this we observe that the metric $g$ and the function $f:=\log u$ satisfy in $M$ the coupled system
\begin{equation}
\label{eq:coupled}
{\rm Ric}_g+\na^2f+\frac{df\otimes df}{n-2}
\,\, = \,\,\,\frac{|\na f|^2_{ g}}{n-2} \, g\qquad\mbox{and}\qquad\Delta_gf=0,
\end{equation}
consisting of a quasi-Einstein type equation and the prescription that $f$ is $g$-harmonic.
We refer the reader to Section \ref{sec:reformulation} for the notation and the computations pertaining to the conformal change.
It is interesting to notice how this formulation suggests a connection between this problem and the static vacuum Einstein equations in General Relativity
\begin{equation*}
f\,{\rm Ric}_g-\na^2f\,\,=\,\,0
\qquad\mbox{and}\qquad\Delta_gf=0,
\end{equation*}
for which analogous rigidity issues arise. To give an example, we just mention the well-known Black Hole Uniqueness Theorem (see \cite{Anderson_book} for a mathematical presentation).
  
Combining the relations in (\ref{eq:coupled}) with the Bochner identity, we deduce that $|\na f|_g^2$ obeys the PDE  
\[
\Delta_g|\na f|_g^2+\big\langle\na|\na f|^2_g\,\big|\,\na f\big\rangle_g
=\,\,2|\na^2f|^2_g.
\]
We then perform in Section \ref{sec:int_per_parti} an analysis of this equation in several situations with the help of different overdetermining conditions and we conclude in each case that $|\na f|_g^2$ is constant and in turn that $\na f$ is a nontrivial parallel vector field. This implies that $g$ has the desired product structure, by a general Riemannian splitting principle. This is described in Section \ref{sec:dim_thm}, where all our results are finally proved. 
We notice en passant that $f$ is an affine function of the $g$-distance to $\pa M$.  
This gives a geometric interpretation of the quantity 
\begin{equation}
\label{eq:gradienti_u_f}
\frac{|\DDD u|^2}{u^{2\frac{n-1}{n-2}}}=|\na f|^2_g\,,
\end{equation}
which is known in the literature as the $P$-function associated with the electrostatic potential (see \cite{Pay_Phi_1}).
In other words, the square root of the $P$-function is the gradient-length of the $g$-distance to $\pa M$, up to a constant factor. 

The second step of our strategy amounts to showing that the fibers of the product are isometric to constant curvature spheres.
Here the key ingredient is the fact that the Weyl tensor of the metric $g$ is always zero, so that we can argue in the same spirit as in \cite{Ca_Ma_Ma_Ri}, where a general study of locally conformally flat quasi-Einstein manifolds is addressed. It is worth mentioning that similar results were first obtained for locally conformally flat gradient Ricci solitons in \cite{Cao_Chen,Cat_Man}.

We now pass to enlist in more detail the results which we have obtained implementing the technique described above.
We are aware of the fact that the overdetermining conditions we deal with are not optimal, but at this stage our aim is more to 
illustrate in a case study a symmetrization technique, 
which is not based neither on the Alexandroff-Serrin method (see \cite{Serrin_ARMA,Reich_1996,Reich_ARMA,Af_Bu_ARMA}) nor on the isoperimetric inequality (see \cite{Pay_Sch,Ga_Sa,Enc_Per}), and which gives some geometrical insights into the method of integral identities introduced by Weinberger, Payne, and Philippin
(see \cite{Wei_ARMA,Pay_Phi_1}). 
We believe that our approach can be applied to other type of equations and boundary conditions, provided that a satisfactory description of the rotationally symmetric solution is available.  
This will be the object of future investigations.
    
\medskip    
We first consider the problem in the exterior domain
\begin{equation}
\label{eq:pb}
\left\{
\begin{array}{rcll}
\displaystyle
-\De u\!\!\!\!&=&\!\!\!\!0 & {\rm in }\quad\R^n\setminus\overline\Om,\\
\displaystyle
  u\!\!\!\!&=&\!\!\!\!1  &{\rm on }\ \ \pa\Om,\\
\displaystyle
u(x)\!\!\!\!&\to&\!\!\!\!0 & \mbox{as }\ |x|\to\infty,
\end{array}
\right.
\end{equation}
where the Dirichlet boundary condition has been normalized to $1$ for sake of simplicity.
It is worth pointing out that for every $0<c\leq1$ the level set $\{u=c\}$ is compact, due to the properness of $u$. 
Moreover, as it is noted in Remark \ref{rm:diffeo}, 
we have that $\{u=c\}$ is diffeomorphic to a 
$(n-1)$-dimensional sphere, and thus connected, 
for every $c>0$ sufficiently small.

Our first result shows that the rotational symmetry of the potential 
$u$ can be deduced provided a suitable integral inequality is 
satisfied on some level set of $u$.
To the authors' knowledge, such an overdetermining condition 
as well as its weaker versions (see conditions (\ref{eq:cond_Lore_Virgi_cor_cap}) and (\ref{eq:cond_Lore_Virgi_cor}) below)
have not yet been considered in the literature.

\begin{theorem}
\label{thm:Lore_Virgi}
Let $u$ be the unique solution to problem (\ref{eq:pb}).
Assume that, for some $0<c\leq1$,
the level set $\{u=c\}$ is regular and it holds
\begin{equation}
\label{eq:cond_Lore_Virgi}
\int\limits_{\{u=c\}}|\DDD u|^2\left[\frac{\HHH}{n-1}-
\frac{|\DDD u|}{(n-2)\,u}\right]\rmd\sigma\leq0,
\end{equation}
where $\HHH$ is the mean curvature of the hypersurface $\{u=c\}$ with respect to the normal pointing towards the interior of $\{u\leq c\}$. 
Then $\Om$ is a ball and $u$ is rotationally symmetric.
\end{theorem}

Let us briefly mention that the regularity of the level set
required in the previous statement is a condition
which ensures that the normal is well-defined all over the 
hypersurface and allows us to safely use the Divergence Theorem 
in the proof. 
However, while we refer the reader to Remark \ref{rm:diffeo}
for some comments about the regularity of the level sets of $u$,
we point out that the thesis of this theorem and of
the related corollaries
still holds when the level set considered is not regular.
We will address this issue with more detail in future work.
Also, let us note that we do not require the connectedness of the
level set where condition \eqref{eq:cond_Lore_Virgi} is fulfilled. 
Indeed, a perusal of the proof shows that
our argument does not rely on the connectedness of
$\{u=c\}$ and, in particular, on the connectedness of $\pa\Om$.
This information can be deduced a posteriori, 
as a consequence of \eqref{eq:cond_Lore_Virgi}.
Indeed, this integral condition is exploited in the proof
to show that $|\na f|_g$ is a positive constant and hence, in view
of \eqref{eq:gradienti_u_f}, that all the level sets $\{u=t\}$
for $0<t\leq c$ are regular. In particular, these level sets
are diffeomorphic to each other and hence connected,
since $\{u=c\}$ is diffeomorphic to a sphere for 
every $c>0$ small enough, as already observed.
 
By means of the coarea formula and Sard's Lemma, one can show that
an analogous ``sphere theorem'' holds if 
the slice-wise condition (\ref{eq:cond_Lore_Virgi}) 
is replaced by a global integral condition. 
This is the content of the following corollary.

\begin{corollary}
\label{cor:cond_glob}
Let $u$ be the unique solution to problem (\ref{eq:pb})
and suppose that 
\begin{equation}
\label{eq:cond_Phi}
\frac{\Phi(1)}{\displaystyle\int_0^1\Phi(c)\,\rmd c}
\,\leq\,
2\,\Big(\frac{n-1}{n-2}\Big),
\qquad
\mbox{with}
\quad
\Phi(c):=\int\limits_{\{u=c\}}\!\frac{|\DDD u|^3}u\,\rmd\sigma.
\end{equation} 
Then $\Om$ is a ball and $u$ is rotationally symmetric.
\end{corollary}

To give a physical interpretation of condition (\ref{eq:cond_Lore_Virgi}), we recall that the \emph{electrostatic capacity} or \emph{Newtonian capacity} of $\Om$ is defined, according to \cite[Section 2.9]{Gil_Tru_book}, as
\begin{equation*}
{\rm Cap}(\Om):=\,\inf\bigg\{\,\int_{\R^n}|\DDD w|^2\,{\rm d}\mu\,:\,w\in C_{c}^{\infty},\,w\equiv1\mbox{ in }\Om\bigg\}.
\end{equation*}
It is easy to see that the electrostatic capacity can be equivalently expressed in terms of the solution $u$ to problem (\ref{eq:pb}) as 
\begin{equation*}
{\rm Cap}(\Om)\,=\!\!\int\limits_{\R^n\setminus\Om}|\DDD u|^2\,{\rm d}\mu=\int\limits_{\pa\Om}|\DDD u|\,{\rm d}\sigma
=\!\!\!\int\limits_{\{u=c\}}\!\!\!|\DDD u|\,{\rm d}\sigma,
\end{equation*}
for every $0<c<1$. Note that the Dirichlet integral of $u$ is finite thanks to Lemmata \ref{lem:u_behaves} and \ref{lem:u_deriv_behaves} and that the second and the third equality are due to 
the Divergence Theorem.
In the following corollary, which is an immediate consequence of Theorem \ref{thm:Lore_Virgi}, we replace condition (\ref{eq:cond_Lore_Virgi}) by a condition involving ${\rm Cap}(\Om)$ and the \emph{total mean curvature} of $\pa\Om$, which is defined as
$
\int_{\pa\Om}\HHH/(n-1)\,{\rm d}\sigma.
$

\begin{corollary}
\label{cor:Lore_Virgi_cap}
Let $u$ be a solution to problem (\ref{eq:pb}) satisfying the condition 
\begin{equation}
\label{eq:cond_Lore_Virgi_cor_cap}
\Big(\max\limits_{\pa\Om}|\DDD u|^2\Big/\min\limits_{\pa\Om}|\DDD u|^2\Big)
\!\!\int\limits_{\pa\Om}\frac{\HHH}{n-1}\,{\rm d}\sigma
\,\leq\,
\frac{{\rm Cap}(\Om)}{n-2}.
\end{equation}
Then $\Om$ is a ball and $u$ is rotationally symmetric.
\end{corollary}

To give another interpretation of condition (\ref{eq:cond_Lore_Virgi}), note that, a posteriori, this becomes an equality.  
More precisely, if $u$ is rotationally symmetric, one has that
\begin{equation}
\label{eq:H=}
\frac{\HHH}{n-1}=\frac{|\DDD u|}{(n-2)\,u},
\end{equation}
on every level set of $u$. 
Thus, the geometrical character of condition (\ref{eq:cond_Lore_Virgi}) is better understood by its pointwise version, which gives rise to the following corollary.

\begin{corollary}
\label{cor:Lore_Virgi}
Let $u$ be a solution to problem (\ref{eq:pb}) satisfying the condition 
\begin{equation}
\label{eq:cond_Lore_Virgi_cor}
\frac{\HHH}{n-1}
\,\leq\,
\frac{|\DDD u|}{(n-2)\,u}
\qquad\mbox{on }\ \{u=c\},
\end{equation}
for some $0<c\leq 1$
such that $\{u=c\}$ is a regular level set.
Then $\Om$ is a ball and $u$ is rotationally symmetric.
\end{corollary}

Let us give one more comment about the geometric nature of (\ref{eq:H=}). In force of identity (\ref{eq:formula_H_H_g}), we have that (\ref{eq:H=}) holds true on some level set of $u$ is and only if such level set is a minimal hypersurface inside the conformally flat Riemannian manifold $(M,g)$.

\medskip
We now pass to
consider a more classical set of overdetermining conditions.
The following theorem tells us that if a solution to problem (\ref{eq:pb}) is such that $|\DDD u|$ is constant on $\{u=c\}$, then the same conclusions as in Theorem \ref{thm:Lore_Virgi} can be derived if we replace $\HHH$ by its infimum on $\{u=c\}$ in (\ref{eq:cond_Lore_Virgi_cor}).
It is worth pointing out that this result has been obtained
under weaker assumptions by several authors. Dropping any 
attempt of being complete, 
we refer the reader to
\cite{Af_Bu_ARMA,Fra_Ga_2008,Ga_Sa,Reich_1996,Reich_ARMA}.  

\begin{theorem}
\label{thm:nabla_const}
Let $u$ be a solution to problem (\ref{eq:pb}), satisfying the conditions
\begin{equation}
\label{eq:hyp_nabla_const}
|\DDD u|=d\quad\ \mbox{ on }\ \{u=c\}\qquad\mbox{ and }
\qquad
\inf_{\{u=c\}}\frac{\HHH}{n-1}\leq\frac{|\DDD u|}{(n-2)\,u}
\quad\ \mbox{on }\ \{u=c\},
\end{equation}
for some $0<c\leq 1$ and some $d>0$.
Then $\Om$ is a ball and $u$ is rotationally symmetric.
\end{theorem}

Hypothesis (\ref{eq:hyp_nabla_const}) on $\pa\Om$ has been considered in \cite{Fra_Ga_2008} to prove rotational symmetry of the solutions in the case of \emph{partially} overdetermined elliptic problems.
To give an explanation of the second condition in (\ref{eq:hyp_nabla_const}), we observe that, in the same spirit of 
\cite{Ga_Sa}, one can deduce from the constancy of the gradient on $\pa\Om$ that either $\HHH/n-1=|\DDD u|/(n-2)\,u$ 
or $\HHH/n-1>|\DDD u|/(n-2)\,u$ on $\{u=c\}$.
While the first case allows us to conclude, the second case can be excluded thanks to the second condition in (\ref{eq:hyp_nabla_const}) or, as done in \cite{Ga_Sa}, thanks to the assumption that $\Om$ is starshaped.
A similar comment can be made about the corresponding condition in
Theorem \ref{thm:nabla_const_delta} below.
Finally, we point out that, 
as well as for Theorem \ref{thm:Lore_Virgi} and related corollaries,
the connectedness of $\Om$ is 
not required and can be deduced a posteriori.

\medskip
  
We now turn the attention to the second problem in (\ref{eq:problemi}).
In other words, we consider a solution of the Laplace equation in the punctured domain, with constant Dirichlet boundary data and with a non-removable singularity at the origin.
As pointed out in \cite{Serrin_1965},
note that the first and the third condition of this problem imply that $-\De u=d|\pa\Om|\delta_0$ in $\Om$, 
for some $d>0$, where $\delta_0$ is the Dirac distribution centred at $0$ and the equation is understood in the sense of distributions. 
For this reason, in what follows we will always deal with the problem 
\begin{equation}
\label{eq:pb_delta}
\left\{
\begin{array}{rcll}
\displaystyle
-\Delta u\!\!\!\!&=&\!\!\!\!d\,|\pa\Om|\,\delta_0 & {\rm in }\quad\Om,\\
\displaystyle
  u\!\!\!\!&=&\!\!\!\!c  &{\rm on }\ \ \pa\Om,\\
\end{array}
\right. 
\end{equation}
for some given $d>0$ and $c\in\R$.
For future reference, let us remark that the first equation in the above problem 
implies that
\begin{equation}
\label{eq:cond_nec}
\int_{\pa\Om}|\DDD u|\,\rmd\sigma=d\,|\pa\Om|.
\end{equation}
This can be easily deduced via the Divergence Theorem and Lemma \ref{lem:delta_behaves}.

The next results are the counterpart of Theorem \ref{thm:Lore_Virgi},
Corollary \ref{cor:Lore_Virgi}, and Theorem \ref{thm:nabla_const} in the
interior punctured domain.
To the authors' knowledge, the following overdetermining condition 
(\ref{eq:cond_Lore_Virgi_delta}) and its weaker version
(\ref{eq:cond_Lore_Virgi_cor_delta})
have not been considered in the literature so far.

\begin{theorem}
\label{thm:Lore_Virgi_delta}
Let $u$ be a solution to problem (\ref{eq:pb_delta}) satisfying the integral condition 
\begin{equation}
\label{eq:cond_Lore_Virgi_delta}
\frac{\displaystyle\fint_{\pa\Om}\frac{\HHH}{n-1}\,\,|\DDD u|^2\,\rmd\sigma}
     {\displaystyle\left(\fint_{\pa\Om}|\DDD u|\,\rmd\sigma\right)^2}
\geq
\left(\frac{|\Sph^{n-1}|}{|\pa\Om|}\right)^{\frac1{n-1}}
\left[
     \frac{\displaystyle\fint_{\pa\Om}|\DDD u|^3\,\rmd\sigma}
     {\displaystyle\left(\fint_{\pa\Om}|\DDD u|\,\rmd\sigma\right)^3}
\right]^{\frac{n}{2(n-1)}}\!\!\!\!\!\!\!\!\!\!\!\!,
\end{equation}
where $\HHH$ is the mean curvature of the hypersurface $\pa\Om$ with respect to the normal pointing towards the exterior of $\Om$. 
Then $\Om$ is a ball and $u$ is rotationally symmetric.
\end{theorem}

Note that the quantity $\big(|\Sph^{n-1}|/|\pa\Om|\big)^{\frac1{n-1}}$ appearing in the above integral inequality is the inverse of the $n$-dimensional \emph{Russel capacity} (see \cite{Cra_Fra_Ga} and the references therein).
In the following straightforward corollary of Theorem \ref{thm:Lore_Virgi_delta} we consider the
pointwise version of hypothesis (\ref{eq:cond_Lore_Virgi_delta}).
This result can be seen as a bridge between Theorem \ref{thm:Lore_Virgi_delta} and Theorem \ref{thm:nabla_const_delta} below.

\begin{corollary}
\label{thm:altro_thm}
Let $u$ be a solution to problem (\ref{eq:pb_delta}) satisfying the condition 
\begin{equation}
\label{eq:cond_Lore_Virgi_cor_delta}
\frac{\HHH}{n-1}
\,\geq\,
\left(\frac{|\Sph^{n-1}|}{|\pa\Om|}\right)^{\frac1{n-1}}
\left[
\frac{\displaystyle\left(\fint_{\pa\Om}|\DDD u|\,\rmd\sigma\right)^2}
{\displaystyle\fint_{\pa\Om}|\DDD u|^2\,\rmd\sigma}
\right]
\!\!
\left[
     \frac{\displaystyle\fint_{\pa\Om}|\DDD u|^3\,\rmd\sigma}
     {\displaystyle\left(\fint_{\pa\Om}|\DDD u|\,\rmd\sigma\right)^3}
\right]^{\frac{n}{2(n-1)}}
\quad\mbox{on }\ \{u=c\}.
\end{equation}
Then $\Om$ is a ball and $u$ is rotationally symmetric.
\end{corollary}

\medskip

In the following theorem we consider the case where a solution to problem (\ref{eq:pb_delta}) has constant Neumann boundary data. In this case, the first equation of (\ref{eq:pb_delta}) and (\ref{eq:cond_nec}) imply that the constant must coincide with $d$, namely
\begin{equation}
\label{eq:Neu_cond}
|\DDD u|=d\qquad\mbox{on}\quad\pa\Om.
\end{equation} 
Let us remark that when this condition is fulfilled, then the 
pointwise inequality (\ref{eq:cond_Lore_Virgi_cor_delta}) 
reduces to 
\[
\frac{\HHH}{n-1}
\,\geq\,
\left(\frac{|\Sph^{n-1}|}{|\pa\Om|}\right)^{\frac1{n-1}}
\qquad\mbox{on }\ \{u=c\}.
\] 
The following results tells us that a weaker version of this 
condition coupled with (\ref{eq:Neu_cond}) 
guarantees the radial symmetry of the solutions.

\begin{theorem}
\label{thm:nabla_const_delta}
Let $u$ be a solution to problem (\ref{eq:pb_delta}) satisfying the conditions
\begin{equation}
\label{eq:hyp_nabla_const_delta}
|\DDD u|=d\quad\ \mbox{ on }\ \pa\Om\qquad\mbox{ and }
\qquad
\sup_{\pa\Om}\frac{\HHH}{n-1}\geq\left(\frac{|\Sph^{n-1}|}{|\pa\Om|}\right)^{\frac1{n-1}}\!\!\!\!\!\!.
\end{equation}
Then $\Om$ is a ball and $u$ is rotationally symmetric.
\end{theorem}

The same conclusions of the above theorem have been derived
under weaker assumptions, e.g., in \cite{Pay_Sch,{Enc_Per}}.
See also \cite{Ag_Ma} where a $2$-dimensional analysis of problem (\ref{eq:pb_delta}) with more general
symmetric-type Neumann conditions is performed. 

We finally state a result which shows how our method applies to problems defined on annular domains. This result can be thought of as a link between what we have obtained for the exterior domain and for the interior punctured domain.

\begin{theorem}
\label{thm:2bordi}
Let $u$ be a positive solution to either problem (\ref{eq:pb}) or problem (\ref{eq:pb_delta}) satisfying the conditions 
\begin{equation*}
\frac{\HHH}{n-1}\geq\frac{|\DDD u|}{(n-2)\,u}
\quad\ \mbox{on }\ \{u=a\}
\qquad\mbox{and}\qquad
\frac{\HHH}{n-1}\leq\frac{|\DDD u|}{(n-2)\,u}
\quad\ \mbox{on }\ \{u=b\},
\end{equation*}
for some $0<a<b$
such that $\{u=a\}$ and $\{u=b\}$ are regular level sets.
Suppose that either $\{u=a\}$ or $\{u=b\}$ is connected. 
Then $\Om$ is a ball and $u$ is rotationally symmetric.
\end{theorem}

More in general, this theorem holds if we replace the pointwise conditions on the level sets $\{u=a\}$ and $\{u=b\}$ with the corresponding integral conditions modeled on (\ref{eq:cond_Lore_Virgi}), as it will be clear from the arguments which are developed in the following sections.
It would be interesting to investigate the relationship between the  two-boundaries overdetermining condition in Theorem \ref{thm:2bordi} and other overdetermining conditions involving two boundaries, such as the ones given in 
\cite{CiMaSa} and \cite{Salani}.

\smallskip

We believe that the approach that we have developed to deal with the classical problems presented in this paper may represent a starting point for future investigations. In particular, we would like to employ the same machinery to treat overdetermined problems for certain types of semilinear and possibly quasilinear equations.

\section{A conformally equivalent formulation of the problems}
\label{sec:reformulation}

Let us consider a solution $u$ to either problem~(\ref{eq:pb}) or problem~(\ref{eq:pb_delta}). Note that in the first case $0<u<1$, whereas in the second case we can suppose that $u>0$, up to an additive constant. We refer the reader to the Appendix for a description of the basic qualitative properties of the solutions.
Here, we present an equivalent formulation of problems~ (\ref{eq:pb}) and (\ref{eq:pb_delta}), which is based on a conformal change of the Euclidean metric.
To set up the notation, we let $M$ be either $\R^n\setminus\Om$ in the case of problem~(\ref{eq:pb}) or $\overline\Om\setminus\{0\}$ in the case of problem~(\ref{eq:pb_delta}).
We denote by $g_{\R^n}$ the flat Euclidean metric of $\R^n$,
and we consider the conformally equivalent metric given by
\begin{equation}
\label{eq:tilde_g}
g:=u^{\frac2{n-2}}g_{\R^n}.
\end{equation}
For future convenience, we set $f:=\log u$, so that the metric $g$ can be equivalently written as
\begin{equation}
\label{eq:tilde_g_con_f}
g={\rm e}^{\frac{2f}{n-2}}g_{\R^n}.
\end{equation}
To proceed, we fix local coordinates $\{x^{\alpha}\}_{\alpha=1}^n$ in $M$
and we use the formulas from 
\cite[pag 42]{Haw_book}
to deduce that
\begin{equation*}
\Gamma_{\alpha\beta}^{\gamma}
=\GGG_{\alpha\beta}^{\gamma}+\frac1{n-2}\bigg(\delta_{\alpha}^{\gamma}\pa_{\beta}f
                      +\delta_{\beta}^{\gamma}\pa_{\alpha}f-g_{\alpha\beta}^{\R^n}\,g_{\R^n}^{\gamma\eta}\,\pa_{\eta}f\bigg),                                                
\end{equation*}
where  $\GGG_{\alpha\beta}^{\gamma}$ and $\Gamma_{\alpha\beta}^{\gamma}$     
are the Christoffel symbols associated with the metric $g_{\R^n}$ and $g$, respectively.
In what follows we will denote by $\DDD$ and $\na$ the covariant derivatives of the metrics $g_{\R^n}$ and $g$, respectively. The symbols $\DDD^2$, $\na^2$ and $\Delta$, $\Delta_g$ will stand for the corresponding Hessians and Laplacian operators.  
Notice that throughout this paper the Einstein summation convention for the sum over repeated indices is adopted. 
As a consequence of the previous formula, we get
\begin{align}
\na^2_{\alpha\beta}w&=\frac{\pa^2w}{\pa x^{\alpha}\pa x^{\beta}}
                                  -\Gamma_{\alpha\beta}^{\gamma}\frac{\pa w}{\pa x^{\gamma}}\nonumber\\
               &=\DDD^2_{\alpha\beta}w-\frac1{n-2}\bigg(\pa_{\alpha}w\,\pa_{\beta}f
                            +\pa_{\beta}w\,\pa_{\alpha}f
                   -\langle\DDD w|\DDD f\rangle\,g_{\alpha\beta}^{\R^n}\bigg),\label{eq:hess_tilde}                   
\end{align}
for any given $w\in{\rm C}^2(M)$. 
By taking the trace of~(\ref{eq:hess_tilde}), we obtain
\begin{equation}
\label{eq:delta_tilde}
\Delta_gw:=g^{\alpha\beta}\,\na^2_{\alpha\beta}w
={\rm e}^{-\frac{2f}{n-2}}\left(\Delta w+\langle\DDD w|\DDD f\rangle\right).
\end{equation}
In particular, one has
$\Delta_gf=0$,
because
\begin{equation}
\label{eq:delta_f}
\Delta f=-|\DDD f|^2.
\end{equation}   
Observe that \cite[Theorem 1.159]{Besse_book} implies that the Ricci tensor $\Ric_g=\RRR_{\alpha\beta}\,dx^{\alpha}\otimes~dx^{\beta}$ of the metric $g$ is related to the one of the metric $g$ by  
\begin{equation*}
\RRR_{\alpha\beta}=\RRR_{\alpha\beta}^{\R^n}-\DDD^2_{\alpha\beta}f+\frac{\pa_{\alpha}f\,\pa_{\beta}f}{n-2}
    -\frac{\Delta f+|\DDD f|^2}{n-2}\,g_{\alpha\beta}^{\R^n}\\
=-\DDD^2_{\alpha\beta}f+\frac{\pa_{\alpha}f\,\pa_{\beta}f}{n-2},
\end{equation*}
where in the second equality we have used 
(\ref{eq:delta_f}) and the fact that $\Ric_{\R^n}=0$.
Therefore,  from (\ref{eq:hess_tilde}) we get
\begin{equation}
\label{eq:tilde_Ric_f}
\Ric_g+\na^2f+\frac{df\otimes df}{n-2}\,=\,\frac{|\na f|^2_g}{n-2}\,g,
\end{equation}
since
$|\na f|^2_g\,g=|\DDD f|^2\,g_{\R^n}$.
We remark en passant that the latter identity implies in particular that
\begin{equation}
\label{eq:P_function}
|\na f|^2_g=\frac{|\DDD u|^2}{u^{2\frac{n-1}{n-2}}}.
\end{equation}
This formula says that the function $|\na f|^2_g$ coincides with the $P$-function usually associated with problems (\ref{eq:pb}) and (\ref{eq:pb_delta}), see \cite[formula (1.5)]{Ga_Sa} and \cite[formula (7)]{Enc_Per}, respectively.

We are now in the position to reformulate problem~(\ref{eq:pb})
as
\begin{equation}
\label{eq:pb_reform}
\left\{
\begin{array}{rcll}
\displaystyle
\phantom{\frac12}\Delta_gf\!\!\!\!&=&\!\!\!\!0 & {\rm in }\quad M,\\
\displaystyle
\Ric_g+\na^2f+\frac{df\otimes df}{n-2}
\!\!\!\!&=&\displaystyle\!\!\!\!\frac{|\na f|^2_g}{n-2}\,g, & {\rm in }\quad M,\\
\displaystyle
 \phantom{\frac12}f\!\!\!\!&=&\!\!\!\!0  &{\rm on }\ \ \pa M,\\
\displaystyle
\phantom{\frac12}f(x)\!\!\!\!&\to&\!\!\!\!-\infty & \mbox{as }\ x\to\infty,
\end{array}
\right.
\end{equation}
and problem~(\ref{eq:pb_delta}) as
\begin{equation}
\label{eq:pb_2_reform}
\left\{
\begin{array}{rcll}
\displaystyle
\phantom{\frac12}\Delta_gf\!\!\!\!&=&\!\!\!\!0 & {\rm in }\quad M,\\
\displaystyle
\Ric_g+\na^2f+\frac{df\otimes df}{n-2}
\!\!\!\!&=&\displaystyle\!\!\!\!\frac{|\na f|^2_g}{n-2}\,g, & {\rm in }\quad M,\\
\displaystyle
 \phantom{\frac12}f\!\!\!\!&=&\!\!\!\!\log c &{\rm on }\ \ \pa M.\\
\displaystyle
\phantom{\frac12}f(x)-\log\frac{d\,|\pa\Om|\,|x|^{2-n}}{(n-2)|\Sph^{n-1}|}\!\!\!\!&\to&\!\!\!\!0 & \mbox{as }\ x\to0.
\end{array}
\right.
\end{equation}
Note that the limit in (\ref{eq:pb_2_reform}) has been deduced from Lemma \ref{lem:delta_behaves}.
Taking the trace of the second equation in (\ref{eq:pb_reform}) or in (\ref{eq:pb_2_reform}) gives the expression for the scalar curvature $\RRR_g$ of the metric $g$, namely
\begin{equation*}
\frac{\RRR_g}{n-1}=\frac{|\na f|^2_g}{n-2}.
\end{equation*}
For the forthcoming analysis it is important to study the geometry of the level sets of $f$, which coincide with the level sets of $u$ by definition. To this end, we fix on $M$ the $g_{\R^n}
$-unit vector field $\nu:=-\DDD u/|\DDD u|=-\DDD f/|\DDD f|$ and the $g$-unit vector field $\nu_g:=-\na u/|\na u|_g=-\na f/|\na f|_g$. Consequently, the second fundamental forms of the regular level sets of $u$ or $f$ with respect to the flat ambient metric and the conformally-related ambient metric $g$, are given by
\begin{align*}
h(X,Y)&=-\frac{\DDD^2u(X,Y)}{|\DDD u|}=-\frac{\DDD^2f(X,Y)}{|\DDD f|},\\
h_g(X,Y)&=-\frac{\na^2u(X,Y)}{|\na u|_g}=-\frac{\na^2f(X,Y)}{|\na f|_g},
\end{align*}  
respectively, where $X$ and $Y$ are vector fields tangent to the level sets. In particular, $h$ and $h_g$ are related by
\begin{equation*}
h_g(X,Y)={\rm e}^{\frac f{(n-2)}}\Big(h(X,Y)-\frac{|\DDD f|}{n-2}\langle X|Y\rangle\Big).
\end{equation*}
As a consequence, taking into account the fact that $u$ is $g_{\R^n}$-harmonic and $f$ is $g$-harmonic, we have that
the mean curvatures are given by
\begin{equation}
\label{eq:formula_curvature}
\HHH=\frac{\DDD^2 u(\nu,\nu)}{|\DDD u|},
\qquad\qquad
\HHH_g=\frac{\na^2 f(\nu_g,\nu_g)}{|\na f|_g}.
\end{equation}
They are related by the formula
\begin{equation}
\label{eq:formula_H_H_g}
\frac{\HHH_g}{n-1}={\rm e}^{-\frac{f}{n-2}}
\bigg(\frac{\HHH}{n-1}-\frac{|\DDD f|}{n-2}\bigg).
\end{equation}
Having this at hand, we notice that 
the quantity appearing in hypothesis (\ref{eq:cond_Lore_Virgi}) of Theorem \ref{thm:Lore_Virgi} 
can be rewritten as
\begin{equation}
\label{eq:ipotesi_riscritta}
\int\limits_{\{u=c\}}|\DDD u|^2
\left[\frac{\HHH}{n-1}-\frac{|\DDD u|}{(n-2)\,u}\right]\rmd\sigma
\,=\,
c^{\frac n{n-2}}\!\!\!\!\!\!\int\limits_{\{f=\log c\}}|\na f|_g^2\,\frac{\HHH_g}{n-1}\,\rmd\sigma_g.
\end{equation}
As it will be clear at a later stage, this quantity turns out to be relevant also in the proof of Theorem \ref{thm:Lore_Virgi_delta}, through hypothesis (\ref{eq:cond_Lore_Virgi_delta}).


\section{Some consequences of the Bochner formula}
\label{sec:int_per_parti}

For what follows it is useful to specialize the Bochner formula
\[
\frac12\Delta_g|\na f|_g^2=|\na^2f|^2_g+\Ric_g(\na f,\na f)+\big\langle\na(\Delta_gf)\big|\na f\big\rangle_g
\]
to our problems. In particular, from equation (\ref{eq:tilde_Ric_f}) we get
\begin{equation}
\label{eq:Bochner_g}
\Delta_g|\na f|_g^2
=2|\na^2f|^2_g-\big\langle\na|\na f|^2_g\,\big|\,\na f\big\rangle_g.
\end{equation}
In this section we exploit some concequences of this formula in relation to some overdetermining conditions under which the solutions to problems (\ref{eq:pb_reform}) and (\ref{eq:pb_2_reform}) can be proved to be affine.

We start by discussing the integral-type conditions, which
involve the right-hand side of (\ref{eq:ipotesi_riscritta}).
Integrating identity (\ref{eq:Bochner_g}) by parts with respect to a weighted measure we obtain the following general formula.

\begin{lemma}
Given $a$, $b\in f(M)$ such that $a<b$, and any $\varphi:f(M)\to\R$, 
we have that
\begin{align}\label{eq:fi_ganza}
2\int\limits_{\{a<f<b\}}|\na^2f|_g^2 \,\, {\rm e}^{\varphi(f)}\,{\rm d}\mu_g=&
\int\limits_{\{a<f<b\}}\left[\varphi''(f)+(\varphi'(f))^2-\varphi'(f)\right]|\na f|_g^4\,\,{\rm e}^{\varphi(f)}\,{\rm d}\mu_g\nonumber
\\
&+[1-\varphi'(b)]{\rm e}^{\varphi(b)}\!\!\!\int\limits_{\{f=b\}}  |\na f|_g^3\,{\rm d}\sigma_g   
\,+\,{\rm e}^{\varphi(b)}\!\!\!\!\!\int\limits_{\{f=b\}}2|\na f|_g^2\,\HHH_g\,{\rm d}\sigma_g
\nonumber\\
&-[1-\varphi'(a)]{\rm e}^{\varphi(a)}\!\!\!\int\limits_{\{f=a\}}|\na f|^3_g\,{\rm d}\sigma_g
\,-\,{\rm e}^{\varphi(a)}\!\!\!\!\!\int\limits_{\{f=a\}}2|\na f|_g^2\,\HHH_g\,{\rm d}\sigma_g
\end{align}
\end{lemma}

\begin{proof}
First, multiplying equation (\ref{eq:Bochner_g}) by ${\rm e}^{\varphi(f)}$ and integrating by parts twice on the set $\{a<f<b\}$, one gets
\begin{multline}\label{eq:per_fi_ganza}
2\int\limits_{\{a<f<b\}}|\na^2f|_g^2 \,\, {\rm e}^{\varphi(f)}\,{\rm d}\mu_g=
\int\limits_{\{a<f<b\}}\left[\varphi''(f)+(\varphi'(f))^2-\varphi'(f)\right]|\na f|_g^4\,\,{\rm e}^{\varphi(f)}\,{\rm d}\mu_g
\\
+
\!\!\!\!\!\int\limits_{\pa\{a<f<b\}}\left[1-\varphi'(f)\right]|\na f|_g^2\,\,\left\langle\na f\,|\,\nu\right\rangle_g\,\,{\rm e}^{\varphi(f)}\,{\rm d}\sigma_g
\,-
\!\!\!\!\!\int\limits_{\pa\{a<f<b\}}\left\langle\na|\na f|_g^2\,|\,\nu\right\rangle_g\,\,{\rm e}^{\varphi(f)}\,{\rm d}\sigma_g,
\end{multline}
where $\nu$ is the outer unit normal of $\{a<f<b\}$ on $\pa\{a<f<b\}$.
Noticing that $\nu=\na f/|\na f|_g$ on $\{f=b\}$ and $\nu=-\na f/|\na f|_g$ on $\{f=a\}$, and using the formula 
\begin{equation*}
\left\langle\na|\na f|_g^2\,|\,\na f\right\rangle_g=2|\na f|^3\,\HHH_g,
\end{equation*}
equation (\ref{eq:fi_ganza}) easily follows from (\ref{eq:per_fi_ganza}).
\end{proof}

If we choose the weighting function $\varphi$ so that
$\varphi''+(\varphi')^2-\varphi'=0$, then the quantity
$(1-\varphi'){\rm e}^{\varphi}$ is a first integral and we obtain the following corollary.

\begin{corollary}
\label{cor:int_pimo}
If $\varphi:I\subseteq f(M)\to\R$ is a solution to $\varphi''+(\varphi')^2-\varphi'=0$ defined on the maximal interval of existence $I$, then, setting
\begin{equation*}
K:=(1-\varphi'){\rm e}^{\varphi},
\end{equation*}
formula (\ref{eq:fi_ganza}) particularizes to
\begin{align}
\label{eq:fi_ganza_partic}
2\int\limits_{\{a<f<b\}}|\na^2f|_g^2 \,\, {\rm e}^{\varphi(f)}\,{\rm d}\mu_g\,\,=
&
\,\,\,\,K\,\bigg(\int_{\{f=b\}}  |\na f|_g^3\,{\rm d}\sigma_g-   
\int_{\{f=a\}}|\na f|^3_g\,{\rm d}\sigma_g\bigg)\nonumber\\
&
+{\rm e}^{\varphi(b)}\!\!\!\!\!\int\limits_{\{f=b\}}2|\na f|_g^2\,\HHH_g\,{\rm d}\sigma_g\,-\,{\rm e}^{\varphi(a)}\!\!\!\!\!\int\limits_{\{f=a\}}2|\na f|_g^2\,\HHH_g\,{\rm d}\sigma_g,
\end{align}
for every $a,b\in I$ with $a<b$.
\end{corollary}

Specializing this corollary to a suitable choice of the function 
$\varphi$, with the following two propositions we prove that the harmonic function $f$ in problems (\ref{eq:pb_reform}) and (\ref{eq:pb_2_reform}) is indeed an affine function, provided appropriate integral conditions are satisfied. 

\begin{proposition}
\label{prop:int_part_1}
Let $f$ be a smooth function satisfying problem~(\ref{eq:pb_reform}) in $M=\R^n\setminus\overline\Om$, where $g$ is metric defined in (\ref{eq:tilde_g_con_f}). If the condition
\begin{equation}
\label{eq:cond_curv_media_1}
\int\limits_{\{f=\log c\}}|\na f|_g^2\,\HHH_g\,\rmd\sigma_g\leq0,
\end{equation}
is fulfilled for some $0<c\leq1$, then $\na^2f=0$ in $\{f<\log c\}$.
\end{proposition}

\begin{proof}
We consider the function $\varphi(f)=f$, which satisfies the equation $\varphi''+(\varphi')^2-\varphi'=0$ with the first integral $K=(1-\varphi'){\rm e}^{\varphi}\equiv0$. Therefore, we can apply Corollary \ref{cor:int_pimo} with $a=\log\ep$ and $b=\log c$, for some small $\ep>0$, obtaining
\begin{equation*}
\int\limits_{\{\log\ep\,<f<\log c\}}\!\!\!\!\!\!|\na^2f|_g^2 \,\,{\rm e}^f\,{\rm d}\mu_g
\,=\,
c\!\!\!\!\!\!\int\limits_{\{f=\log c\}}\!\!\!\!\!\!|\na f|_g^2\,\HHH_g\,{\rm d}\sigma_g
\,\,-\,\,\ep\!\!\!\!\!\!\int\limits_{\{f=\log\ep\}}\!\!\!\!\!|\na f|_g^2\,\HHH_g\,{\rm d}\sigma_g.
\end{equation*}   
Next we are going to show that
\begin{equation}
\label{eq:lim_in_prop}
\lim_{\ep\to0}\,\,\,
\ep\!\!\!\!\!\!\!\int\limits_{\{f=\log\ep\}}\!\!\!\!\!|\na f|_g^2\,\HHH_g\,{\rm d}\sigma_g=0.
\end{equation}
Combining this fact with the hypothesis (\ref{eq:cond_curv_media_1}), the statement follows at once.
To check (\ref{eq:lim_in_prop}), we first observe that 
\begin{equation*}
\int\limits_{\{f=\log\ep\}}\!\!\!\!\!|\na f|_g^2\,\HHH_g\,{\rm d}\sigma_g
\,\,=\!\!\!\!\!
\int\limits_{\{f=\log\ep\}}\!\!\!\!\!|\na f|_g\,\na^2f(\nu_g,\nu_g)\,{\rm d}\sigma_g
\,\,\leq\!\!\!\!\!
\int\limits_{\{f=\log\ep\}}\!\!\!\!\!|\na f|_g\,|\na^2f|_g\,{\rm d}\sigma_g,
\end{equation*}
where in the equality we have used (\ref{eq:formula_curvature}).
By Lemmata \ref{lem:u_behaves} and \ref{lem:u_deriv_behaves} and by the identity 
$|\na f|_g=u^{-\frac{n-1}{n-2}}|\DDD u|$,
it is easy to see that there exists a constant $A_1>0$ such that 
\begin{equation}
\label{eq:na_f_bounded}
|\na f|_g\leq A_1.
\end{equation}
Now, using (\ref{eq:hess_tilde}) and making $u$ explicit, we have that
\begin{align}
\big|\na^2f\big|^2_g
&=
\bigg|\frac{\DDD^2u}u-\frac n{n-2}\frac{du\otimes du}{u^2}+\frac1{n-2}\left|\frac{\DDD u}u\right|^2\!\!g_{\R^n}\bigg|_g^2
\nonumber\\
&=
u^{-\frac4{n-2}}\bigg(\left|\frac{\DDD^2 u}u\right|^2+\frac{n(n-1)}{(n-2)^2}\left|\frac{\DDD u}u\right|^4-\,\frac{2n}{n-2}\frac{\DDD^2u(\DDD u,\DDD u)}{u^3}\bigg)
\nonumber\\
&\leq 
A_2\,u^{-\frac4{n-2}}\bigg(\left|\frac{\DDD^2 u}u\right|^2+\left|\frac{\DDD u}u\right|^4\bigg),
\label{eq:per_D^2_bdd}
\end{align}
for some constant $A_2>0$.
Note that in the last inequality we have used the fact that
\[
\left|\frac{\DDD^2u(\DDD u,\DDD u)}{u^3}\right|
\leq
\left|\frac{\DDD^2u}u\right|\left|\frac{\DDD u}u\right|^2
\leq
\frac12\bigg(\left|\frac{\DDD^2u}u\right|^2+\left|\frac{\DDD u}u\right|^4\bigg).
\]
Employing Lemmata \ref{lem:u_behaves} and \ref{lem:u_deriv_behaves}, from (\ref{eq:per_D^2_bdd}) we can deduce that
there exists $A_3>0$ such that
\begin{equation}
\label{eq:D^2_bdd}
|\na^2f|_g\leq A_3.
\end{equation}
Finally, observe that
\begin{equation*}
\int\limits_{\{f=\log\ep\}}\!\!{\rm d}\sigma_g
\,=\!\!\!
\int\limits_{\{u=\ep\}}\!u^{\frac{n-1}{n-2}}\,{\rm d}\sigma
\,=\,\,
\ep^{\frac{n-1}{n-2}}\!\!
\int\limits_{\{u=\ep\}}\!\!{\rm d}\sigma.
\end{equation*}
Using Lemma \ref{lem:u_behaves}, it is not hard to show that
$|\{u=\ep\}|\leq A_4\,\ep^{-(n-1)/(n-2)}$, for some $A_4>0$, so that
\begin{equation}
\label{eq:vol_bdd}
\int\limits_{\{f=\log\ep\}}\!\!{\rm d}\sigma_g
\leq
A_4.
\end{equation}
Inequalities (\ref{eq:na_f_bounded}), (\ref{eq:D^2_bdd}), and (\ref{eq:vol_bdd}) imply (\ref{eq:lim_in_prop}).
This concludes the proof. 
\end{proof}

The following proposition is the counterpart of the previous one in the interior punctured domain. In this case, since $f(x)\to+\infty$ as $x\to0$, the integrability of the function
\[
|\na^2 f|^2{\rm e}^{\varphi(f)}
\]
on $\{f>\log c\}$ is not a priori guaranteed by the choice 
$\varphi(f)=f$, so that $\varphi$ has to be chosen differently.
The new determination of $\varphi$ implies that the first integral $K$ is in general different from $0$. This justify the introduction of an extra hypothesis on the limiting behavior of a relevant integral quantity. 

\begin{proposition}
\label{prop:int_part_2}
Let $f$ be a smooth function satisfying problem~(\ref{eq:pb_2_reform}) in $M=\Om\setminus\{0\}$, where $g$ is the metric defined in (\ref{eq:tilde_g_con_f}). If the conditions
\begin{equation}
\label{eq:cond_curv_media_2}
\int\limits_{\{f=\log c\}}\!\!\!|\na f|_g^3\,{\rm d}\sigma_g
\,
=
\,\,
\limsup_{t\to\infty}\!\!\!\int\limits_{\{f=\log t\}}\!\!\!|\na f|_g^3\,{\rm d}\sigma_g
\qquad\mbox{ and }\quad
\int\limits_{\{f=\log c\}}|\na f|_g^2\,\HHH_g\,\rmd\sigma_g\geq0
\end{equation}
are fulfilled, 
then $\na^2f=0$ in $\{f>\log c\}$.
\end{proposition}

\begin{proof}
For $c<t<+\infty$, we consider the family of functions
\begin{equation*}
\varphi_t(f):=\log\big(1-{\rm e}^f\!/t\big),
\end{equation*} 
defined for $\log c<f<\log t$.
Since $\varphi_t$ satisfies the equation $\varphi''+(\varphi')^2-\varphi'=0$ with the first integral $K=(1-\varphi_t'){\rm e}^{\varphi_t}\equiv1$, we can apply Corollary \ref{cor:int_pimo} with $a=\log c$ and $b=\log(t-\ep)$, for some small $\ep>0$, obtaining
\begin{multline*}
2\!\!\!\!\!\!\int\limits_{\{\log c\,<f<\log(t-\ep)\}}\!\!\!\!\!\!\!\!\!|\na^2f|_g^2 \,\,\big(1-{\rm e}^f\!/t\big)\,{\rm d}\mu_g
\,\,=\!\!\!\!\int\limits_{\{f=\log(t-\ep)\}}\!\!\!\!\!\!\!\!\!|\na f|_g^3\,{\rm d}\sigma_g\,-\!\!\!   
\int\limits_{\{f=\log c\}}\!\!\!\!\!|\na f|^3_g\,{\rm d}\sigma_g\\
+(\ep/t)\!\!\!\!\!\!\!\!\int\limits_{\{f=\log(t-\ep)\}}\!\!\!\!\!\!\!\!\!2|\na f|_g^2\,\HHH_g\,{\rm d}\sigma_g\,-\,(1-c/t)\!\!\!\!\!\int\limits_{\{f=\log c\}}\!\!\!\!\!2|\na f|_g^2\,\HHH_g\,{\rm d}\sigma_g.
\end{multline*}   
Letting $\ep\to0$ in the previous formula, we end up with
\begin{multline}
\label{eq:per_limite_t}
2\!\!\!\!\!\!\int\limits_{\{\log c\,<f<\log t\}}\!\!\!\!\!\!\!\!\!|\na^2f|_g^2 \,\,\big(1-{\rm e}^f\!/t\big)\,{\rm d}\mu_g
\,\,=\!\!\!\int\limits_{\{f=\log t\}}\!\!\!|\na f|_g^3\,{\rm d}\sigma_g\,-\!\!\!   
\int\limits_{\{f=\log c\}}\!\!\!|\na f|^3_g\,{\rm d}\sigma_g\\
\,-\,(1-c/t)\!\!\!\!\!\int\limits_{\{f=\log c\}}\!\!\!\!\!2|\na f|_g^2\,\HHH_g\,{\rm d}\sigma_g.
\end{multline} 
Finally, by the first condition in assumption (\ref{eq:cond_curv_media_2}) and by the Monotone Convergence Theorem,
letting $t\to+\infty$ in (\ref{eq:per_limite_t}) gives
\begin{equation*}
\!\!\!\!\!\!\int\limits_{\{f>\log c\}}\!\!\!\!\!\!\!\!\!|\na^2f|_g^2\,\,{\rm d}\mu_g
\,\,
=
\,\,
-\!\!\!\!\!\int\limits_{\{f=\log c\}}\!\!\!\!\!|\na f|_g^2\,\HHH_g\,{\rm d}\sigma_g.
\end{equation*}
The statement is now a direct consequence of the second condition in assumption (\ref{eq:cond_curv_media_2}).
\end{proof}

We now come back to the analysis of identity (\ref{eq:Bochner_g}), this time in relation to pointwise-type overdetermining conditions
for problems (\ref{eq:pb_reform}) and (\ref{eq:pb_2_reform}). 
In analogy with Proposition \ref{prop:int_part_1} and Proposition \ref{prop:int_part_2}, our conclusion will be that $f$ is an affine function.
We start the analysis of problem~(\ref{eq:pb_reform}) with the following lemma.

\begin{lemma}
\label{lem:Garof_equiv}
Let $f$ be a smooth function satisfying problem~(\ref{eq:pb_reform}) in 
$M=\R^n\setminus\Om$, where $g$ is the metric defined in (\ref{eq:tilde_g_con_f}).
Then, for every $0<c\leq1$,
\begin{equation*}
\max_{\{f\leq\log c\}}|\na f|_g^2
=
\max_{\{f=\log c\}}|\na f|_g^2.
\end{equation*}
\end{lemma}

\begin{proof}
Let us restrict to the case $c\,{=}1$, so that $\{f\leq\log c\}=M$ and $\{f=\log c\}{=}~\pa M$. The proof for an arbitrary $0<c<1$ is identical, up to minor changes. 
We define the functions
\begin{equation*}
w_{\alpha}:=|\na f|_g^2\,{\rm e}^{\alpha f},\qquad\alpha\in(0,1).
\end{equation*}
Using equation (\ref{eq:Bochner_g}), one can show that the functions $w_{\alpha}$ satisfy
\begin{equation}
\label{eq:eq_alpha}
\Delta_gw_{\alpha}+(1-2\alpha)\langle\na w_{\alpha}|\na f\rangle_g
-\alpha(1-\alpha)|\na f|_g^2\,w_{\alpha}
=
2\,{\rm e}^{\alpha f}|\na^2f|^2_g.
\end{equation}
Observe that due to Lemmata \ref{lem:u_behaves} and
\ref{lem:u_deriv_behaves} and equation (\ref{eq:P_function}), we have that $w_{\alpha}(x)\to0$, as $|x|\to\infty$.
This fact, together with equation (\ref{eq:eq_alpha}), tells us that the function $w_{\alpha}$ attains its maximum value on $\pa M$. Indeed, the equation prevents $w_{\alpha}$ from attaining the maximum in the interior of $M$, unless $w_{\alpha}$ is a constant function. But since $w_{\alpha}$ vanishes at infinity, if $w_{\alpha}$ were constant we would have $0=w_{\alpha}=|\na f|_g$ on $\pa M$, which is impossible.
Now, let $p_{\alpha}\in\pa M$ be such that $\max_{\overline M}w_{\alpha}=w_{\alpha}(p_{\alpha})$ so that, fixed an arbitrary $x\in M$, we have
\begin{equation*}
w_{\alpha}(p_{\alpha})\geq w_{\alpha}(x).
\end{equation*}
Therefore, up to a subsequence, we have $p_{\alpha}\to p_0$, as $\alpha\to0^+$, for some $p_0\in\pa M$, and in turn $|\na f|_g(p_0)\geq|\na f|_g(x)$. Since this holds true for every $x\in M$, the thesis follows.
\end{proof}


Combining the above lemma with the Strong Maximum Principle we can deduce the next result. 

\begin{proposition}
\label{prop:const_hyp}
Let $f$ be a smooth function satisfying problem~(\ref{eq:pb_reform}) in $M=\R^n\setminus\Om$, where $g$ is the metric defined in (\ref{eq:tilde_g_con_f}). If the conditions
\begin{equation}
\label{eq:const_hyp}
|\na f|_g=d\ \ \mbox{ on $\ \pa M$}\qquad\mbox{ and }\qquad\inf_{\pa M}\HHH_g\leq0,
\end{equation}
for some $d>0$, then $\na^2f\equiv0$ in $M$.
Moreover, for every $0<c<1$, the same conclusion holds in $\{f\leq\log c\}$ if condition (\ref{eq:const_hyp}) is satisfied with $\{f=\log c\}$ in place of $\pa M$.
\end{proposition}

\begin{proof}
We first note that $\HHH_g\geq0$ on $\pa M$. This follows from (\ref{eq:formula_curvature}) together with the fact that 
$
2\na^2f(\na f,\na f)=\langle\na|\na f|_g^2|\na f\rangle_g\geq0,
$
which is due to Lemma \ref{lem:Garof_equiv} coupled with the first condition in (\ref{eq:const_hyp}). 
Now, using \cite[Lemma 3.4]{Gil_Tru_book} and again the first condition in (\ref{eq:const_hyp}), we have that either $\langle\na|\na f|_g^2|\na f\rangle_g>0$ on $\pa M$ or $|\na f|_g$ is constant in $M$. Observe that the second condition of (\ref{eq:const_hyp}) and (\ref{eq:formula_curvature}) prevent the first possibility to occur, because of the compactness of $\pa M$. Therefore, we have deduced that $|\na f|_g$ is constant in $M$ and in turn, from (\ref{eq:Bochner_g}), that $\na^2f=0$ in $M$. 
The second part of the statement follows by the same argument, 
up to minor changes.  
\end{proof}

Note that, in virtue of Lemma \ref{lem:Garof_equiv}, the constant $d$ appearing in the previous statement must coincide with $\max_M|\na f|_g$. The following result can be seen as the counterpart of Lemma \ref{lem:Garof_equiv} and Proposition \ref{prop:const_hyp} for the interior punctured domain.

\begin{proposition}
\label{prop:const_hyp_2}
Let $f$ be a smooth function satisfying problem~(\ref{eq:pb_2_reform}) in $M=\overline\Om\setminus\{0\}$, where $g$ is the metric defined in (\ref{eq:tilde_g_con_f}). If the conditions
\begin{equation}
\label{eq:const_hyp_2}
|\na f|_g
\,\equiv\,\,
\max_M|\na f|_g\ \ \mbox{ on $\ \pa M$}\qquad\mbox{ and }\qquad\sup_{\pa M}\HHH_g\geq0
\end{equation}
are fulfilled, then $\na^2f=0$ in $M$.
\end{proposition}

\begin{proof}
Note that from the first condition in (\ref{eq:const_hyp_2}) we obtain that
$
2\na^2f(\na f,\na f)=\langle\na|\na f|_g^2|\na f\rangle_g\leq0,
$
so that $\HHH_g\leq0$ on $\pa M$. 
Now, using \cite[Lemma 3.4]{Gil_Tru_book} and again the first condition in (\ref{eq:const_hyp_2}), we have that either $\langle\na|\na f|_g^2|\na f\rangle_g<0$ on $\pa M$ or $|\na f|_g$ is constant in $M$. Observe that the second condition of (\ref{eq:const_hyp_2}) and (\ref{eq:formula_curvature}) prevent the first possibility to occur, because of the compactness of $\pa M$. Therefore, we have deduced that $|\na f|_g$ is constant in $M$ and in turn, from (\ref{eq:Bochner_g}), that $\na^2f=0$ in $M$.   
\end{proof}
 

\section{Product structure of the conformal metric \\
and rotational symmetry}
\label{sec:dim_thm}

This section is devoted to the proofs of Theorem 
\ref{thm:Lore_Virgi}, Theorem~\ref{thm:nabla_const}, 
Theorem~\ref{thm:Lore_Virgi_delta}, Theorem~\ref{thm:nabla_const_delta}, and Theorem \ref{thm:2bordi}.
Moreover, at the end of the section we present the proof of 
Corollary \ref{cor:cond_glob}, which is a consequence of 
Theorem \ref{thm:Lore_Virgi} combined with the coarea formula.
 
We first show that if an affine function $f$ solves equation (\ref{eq:tilde_Ric_f}) in a complete noncompact Riemannian manifold 
$(M,g)$ with smooth boundary, then $(M,g)$ has a product structure provided $f$ is constant on $\pa M$. Moreover, if 
the Weyl tensor associated with the metric is null, 
we have that $(M,g)$ is indeed isometric to one half round cylinder. 
Building on this result, we then consider the function $f:=\log u$ satisfying problems (\ref{eq:pb_reform}) and (\ref{eq:pb_2_reform}) and for each of the overdetermining conditions considered in the theorems we prove that the hypersurfaces $\{f=\log c\}$, which coincide with $\{u=c\}$, are isometric to constant curvature spheres. 
The rotational symmetry of $u$ will then follow by standard arguments.

Since the subsequent proposition and its proof are based on the orthogonal decomposition of the Riemann tensor as well as on the \emph{Gauss equations}, we briefly recall these classical formulas for ease of reference.

Given a $n$-dimensional Riemannian manifold,
the Riemann curvature tensor $\Rm$ decomposes as 
\begin{equation}
\label{eq:Rm_decomp}
\RRR_{\alpha\beta\gamma\delta}=\SSS_{\alpha\beta\gamma\delta}+\EEE_{\alpha\beta\gamma\delta}+\WWW_{\alpha\beta\gamma\delta},
\end{equation}
where the pure-trace or scalar part $\SSS$ of $\Rm$ is defined by
\[
\SSS_{\alpha\beta\gamma\delta}:=\frac{\RRR}{n(n-1)}\bigg(g_{\alpha\gamma}g_{\beta\delta}-g_{\alpha\delta}g_{\beta\gamma}\bigg),
\]
the semi-tracelss part $\EEE$ by
\begin{equation}
\label{eq:R_decomp_E}
\EEE_{\alpha\beta\gamma\delta}:=\frac 1{n-2}\bigg(g_{\alpha\gamma}\TTT_{\beta\delta}-g_{\alpha\delta}\TTT_{\beta\gamma}+g_{\beta\delta}\TTT_{\alpha\gamma}-g_{\beta\gamma}\TTT_{\alpha\delta}\bigg),
\end{equation}
where
\begin{equation}
\label{eq:R_decomp_T}
\TTT_{\alpha\beta}:=\RRR_{\alpha\beta}-\frac{\RRR}ng_{\alpha\beta}
\end{equation}
is the traceless Ricci tensor $\stackrel{\circ}{\Ric}$,
and the tensor $\WWW$, the fully traceless part of $\Rm$, is the so-called Weyl tensor.

We recall that an $n$-dimensional Riemannian manifold is a \emph{space form} (or a \emph{constant curvature space})
if $\WWW=0$ and $\EEE=0$.
In this case, by the
Fundamental Classification Theorem (see \cite[Theorem 3.82]{Ga_Fo_book}), 
depending on whether $\RRR>0$, $\RRR=0$, or $\RRR<0$, the manifold is locally isometric either to a round sphere, or to the flat $\R^n$, or to the hyperbolic space, respectively.
In the case where $\EEE=0$, the manifold is said to be an \emph{Einstein manifold}. Notice that $\EEE=0$ if and only if $\stackrel{\circ}{\Ric}=0$. Concerning the  Weyl tensor, it is worth pointing out that for $n\leq3$ the tensor $\WWW$ is not present in the decomposition (\ref{eq:Rm_decomp}), whereas for $n>3$ we have that $\WWW=0$ if and only if the manifold is \emph{locally conformally flat}.

We finally recall the Gauss equations for hypersurfaces.
Let $N$ be a hypersurface sitting in our $n$-dimensional Riemannian manifold, and fix local coordinates $\{t,\theta^1,...,\theta^{n-1}\}$ adapted to the hypersurface in the sense that    $\pa/\pa t$ is a unit normal to $N$ and $\pa/\pa\theta^1,...,\pa/\pa\theta^{n-1}$ are tangent to $N$.    In this coordinates, the Gauss equations read
\begin{align}
\RRR_{ijkl}&=\overline{\RRR}_{ijkl}+h_{jk}h_{il}-h_{ik}h_{jl},\label{eq:Gauss_Riemann}\\
\RRR_{ij}&=\overline{\RRR}_{ij}+\RRR_{itjt}+g^{kl}h_{ik}h_{jl}-\HHH\,h_{ij},\label{eq:Gauss_Ricci}\\
\RRR&=\overline\RRR+2\RRR_{tt}+|h|^2-\HHH^2.\label{eq:Gauss_R}
\end{align}
Here, $\bar g$ is the metric induced by $g$ on the hypersurface, and $\overline\Rm$, $\overline\Ric$, and $\overline\RRR$ are the corresponding  Riemann tensor, Ricci tensor, and scalar curvature.

Having recalled these basic concepts from Riemannian geometry, we are now in the position to prove the following structure theorem.

\begin{theorem}
\label{prop:LCF_sfera}
Let $(M,g)$ be a complete noncompact Riemannian manifold with smooth boundary. Suppose that there exists a nonconstant function $f:M\to\R$ such that the boundary $\pa M$ of $M$ is a level set of $f$. 
We have the following.
\begin{itemize}

\item[(i)] If $\na^2f\equiv0$ in $M$, then 
$(M,g)$ is isometric to the Riemannian product 
\[
\Big([0,+\infty)\times\pa M,dt\otimes dt+g_{|\pa M}\Big),
\]
where the variable $t$ represents the distance to the boundary.

\item[(ii)] If in addition the Weyl tensor $\WWW$ of the metric $g$ vanishes and equation (\ref{eq:tilde_Ric_f}) is satisfied, 
then the $(n-1)$-dimensional Riemannian manifold 
$\big(\pa M,g_{|\pa M}\big)$ is 
locally isometric to a spherical space form and $(M,g)$ is 
locally isometric to one half round cylinder.
If in addition $\pa M$ is connected, the local isometry
improves to a global isometry.
\end{itemize}
\end{theorem}

\begin{proof}
Since in this proof there is no possibility of confusion,
we omit the subscript $g$ to all the quantities.
To prove the first part of the statement, note first that by Kato inequality we have
\[
0=|\na^2 f|\geq|\na|\na f||\qquad\mbox{in}\quad M,
\]
so that $|\na f|$ is constant everywhere. 
Clearly, it cannot be $|\na f|=0$, because $f$ is noncostant. 
Thus, $|\na f|>0$ on $M$ and, in turn, 
the function $f$ can be regarded as a coordinate everywhere and locally one can choose a system of coordinates
$\{f,\theta^1,\ldots,\theta^{n-1}\}$ such that the metric $g$ can be written as
\[
g=\frac{df\otimes df}{|\na f|^2}+g_{ij}(f,\theta)\,d\theta^i\!\otimes d\theta^j.
\]
To obtain the product structure of $g$, it is enough to show that $\pa g_{ij}/\pa f=0$, for $i,j=1,\ldots,n-1$. 
Indeed, setting $dt=\pm\,df/|\na f|$ (depending on whether $\na f$ on $\pa M$ is inward or outward pointing) and normalizing in such a way that the level set $\{t=0\}$ corresponds to $\pa M$, this would imply that
\[
g
=dt\otimes dt+g_{|\pa M}.
\]
Note that this expression tells us that the function $t$ represents the distance to the boundary.
The fact that $g_{ij}$ is independent of $f$ is a consequence of the basic computation
\begin{align*}
0=\na_{ij}^2f
&=
\frac{\pa^2f}{\pa\theta^i\pa\theta^j}-\Gamma_{ij}^f\frac{\pa f}{\pa f}-\Gamma_{ij}^k\frac{\pa f}{\pa\theta^k}=-\Gamma_{ij}^f\\
&=
-\frac{g^{ff}}2\left(\frac{\pa g_{if}}{\pa\theta^j}+\frac{\pa g_{jf}}{\pa\theta^i}-\frac{\pa g_{ij}}{\pa f}\right)
=-\frac{|\na f|^2}2\frac{\pa g_{ij}}{\pa f}.
\end{align*}
This finishes the proof of (i).

To show that the second part of the statement hods true,
we denote by $\bar g$ the metric $g_{|\pa M}$ and we use the same notation for all the Riemannian quantities related to $\bar g$.
We first prove that $\pa M$ is a space form, that is $\overline\EEE=0$ and $\overline\WWW=0$, according to (\ref{eq:Rm_decomp})-(\ref{eq:R_decomp_T}). 
To check that the condition
$\stackrel{\circ}{\overline\Ric}=0$, which is equivalent to $\overline\EEE=0$, is satisfied, let us use the coordinates $\{t,\theta^1,\ldots,\theta^{n-1}\}$ and note that the Gauss equation (\ref{eq:Gauss_Ricci})
reduces to 
\[
\RRR_{ij}=\overline\RRR_{ij}.
\]
Indeed, it is easy to verify that the product structure of $g$ induces a splitting of $\Rm$ as well, so that $\RRR_{itjt}=0$. Moreover, since $\na^2f\equiv0$, then $h\equiv0$ on $\pa M$.
Now, using equation (\ref{eq:tilde_Ric_f}) on $\pa M$, we obtain that for every $i,j=1,\ldots,n-1$, 
\begin{equation}
\label{eq:per_E=0_1}
\overline\RRR_{ij}=-\na^2_{ij}f-\frac{\pa_if\pa_jf}{n-2}+\frac{|\na f|^2}{n-2}g_{ij}
=
\frac{|\na f|^2}{n-2}\bar g_{ij},
\qquad\mbox{on }\ \pa M.
\end{equation}
Taking the trace, we have that indeed 
\begin{equation}
\label{eq:per_E=0_2}
\frac{\overline\RRR}{n-1}=\frac{|\na f|^2}{n-2},
\end{equation}
and in turn, using again (\ref{eq:per_E=0_1}), we have that
$\stackrel{\circ}{\overline\Ric}=0$.
 
To see that $\overline\WWW=0$, observe first that from the decomposition (\ref{eq:Rm_decomp}) applied to $\overline\Rm$ and from the definition of $\overline\SSS$, using $\overline\EEE=0$ we get
\begin{equation}
\label{eq:per_E=0_3}
\overline\WWW_{ijkl}=\RRR_{ijkl}-\frac{\overline\RRR}{(n-1)(n-2)}\bigg(\bar g_{ik}\bar g_{jl}-\bar g_{il}\bar g_{jk}\bigg).
\end{equation}
Observe that we have also used the identity $\overline\RRR_{ijkl}=\RRR_{ijkl}$, which follows from the Gauss equation (\ref{eq:Gauss_Riemann}) and from the fact that $h=0$.
Now, since $\WWW=0$, again from the decomposition of the Riemann tensor and from the definition of $\SSS$ we have 
\begin{equation*}
\RRR_{ijkl}
=
\frac{\RRR}{n(n-1)}\bigg(g_{ik}g_{jl}-g_{il}g_{jk}\bigg)+\EEE_{ijkl}
=
\frac{\overline\RRR}{n(n-1)}\bigg(\bar g_{ik}\bar g_{jl}-\bar g_{il}\bar g_{jk}\bigg)+\EEE_{ijkl}.
\end{equation*}
Note that in the second equality we have used the fact that $\RRR=\overline\RRR$, which is a consequence of (\ref{eq:Gauss_R}) coupled with $\RRR_{tt}=0$.
The last identity and equation (\ref{eq:per_E=0_3}) imply that
\begin{equation}
\label{eq:per_E=0_5}
\overline\WWW_{ijkl}
=
-\frac{2\,\overline\RRR}{n(n-1)(n-2)}\bigg(\bar g_{ik}\bar g_{jl}-\bar g_{il}\bar g_{jk}\bigg)+\EEE_{ijkl}.
\end{equation}
At the same time, equations (\ref{eq:per_E=0_1})-(\ref{eq:per_E=0_2}) yield
$
\TTT_{ij}=\big(\,\overline\RRR/n(n-1)\big)g_{ij},
$
and in turn
\[
\EEE_{ijkl}
=
\frac{2\,\overline\RRR}{n(n-1)(n-2)}\bigg(\bar g_{ik}\bar g_{jl}-\bar g_{il}\bar g_{jk}\bigg).
\]
Substituting this expression into (\ref{eq:per_E=0_5}) gives $\overline\WWW=0$. Hence, $(\pa M,\bar g)$ is a space form. 

In view of identity (\ref{eq:per_E=0_2}), which gives $\overline\RRR>0$, we can finally deduce that $(\pa M,\bar g)$ is locally isometric to a constant curvature sphere. 
\end{proof}

We now proceed with the proofs of the main theorems stated in the introduction. We start with the proof of Theorem \ref{thm:Lore_Virgi}, which treats the problem in the exterior domain with the integral overdetermining condition (\ref{eq:cond_Lore_Virgi}). 

\begin{proof}[Proof of Theorem \ref{thm:Lore_Virgi} (after Proposition
\ref{prop:int_part_1})]
Let $u$ be a solution to problem (\ref{eq:pb}) and let $0<c\leq1$ be such that the hypotheses of the theorem are satisfied.

Let us consider the metric $g$ defined in (\ref{eq:tilde_g}) and the function $f=\log u$, so that the system (\ref{eq:pb_reform}) is satisfied in $M=\R^n\setminus\Om$. According to (\ref{eq:ipotesi_riscritta}), the hypothesis (\ref{eq:cond_Lore_Virgi}) is equivalent to
\[
\int\limits_{\{f=\log c\}}|\na f|_g^2\,\HHH_g\,\rmd\sigma_g\leq0.
\]
Therefore, we can use Proposition \ref{prop:int_part_1} to deduce that $\na^2 f\equiv0$ in $\{f\leq\log c\}$. At the same time, since the metric $g$ is by construction globally conformally equivalent to the flat metric $g_{\R^n}$, we have the for $n\geq4$ the corresponding Weyl tensor $\WWW_g$ vanishes. As for the case $n=3$, we recall that $\WWW_g$ is not present in the decomposition (\ref{eq:Rm_decomp}). Hence, we can apply Theorem \ref{prop:LCF_sfera} and obtain that the Riemannian manifold $\big(\{f=\log c\},g_{|\{f=\log c\}}\big)$, or equivalently $\big(\{u=c\},g_{|\{u=c\}}\big)$, 
is locally isometric to a constant curvature sphere.  
Since the conformal factor which relates $g_{\R^n}$ and $g$ is a function of $u$, it is immediate to see that the metric $g_{\R^n}$ restricted to $\{u=c\}$ is a constant multiple of $\bar g=g_{|\{u=c\}}$. Hence, $\big(\{u=c\},{g_{\R^n}}_{|\{u=c\}}\big)$ is 
locally isometric to a constant curvature sphere as well. 
Observe that from $\na^2 f\equiv 0$ in $\{u\leq c\}$ we deduce 
that $|\na f|_g$ is a positive constant in the same region, 
so that all the level sets $\{u=t\}$, for $t\leq c$, 
are regular and diffeomorphic to each other. 
In particular, in view of Remark \ref{rm:diffeo},
the level sets are all diffeomorphic to $(n-1)$-dimensional spheres.
Therefore, $\big(\{u=c\},{g_{\R^n}}_{|\{u=c\}}\big)$ is 
indeed globally isometric to a constant curvature sphere, hence
we can set $\{u=c\}=\pa B_{r_c}$, for some $r_c>0$.

Let us now show that $u$ is rotationally symmetric in $\R^n\setminus B_{r_c}$. 
To do this, we note that the function 
$
v(x):=c\,(r_c/|x|)^{n-2}
$
satisfies the conditions
\begin{equation*}
\left\{
\begin{array}{rcll}
\displaystyle
\De v\!\!\!\!&=&\!\!\!\!0 & {\rm in }\quad\R^n\setminus\overline B_{r_c},\\
\displaystyle
  v\!\!\!\!&=&\!\!\!\!c  &{\rm on }\ \ \pa B_{r_c},\\
\displaystyle
v(x)\!\!\!\!&\to&\!\!\!\!0 & \mbox{as }\ |x|\to\infty.
\end{array}
\right.
\end{equation*}
On the other hand, the function $u$ solves the same problem.
Hence, by the Maximum Principle, $u=v$ in $\R^n\setminus B_{r_c}$ and in turn $u$ is rotationally symmetric in the same set. 

Finally, we prove that
$u$ is rotationally symmetric on the whole $\R^n\setminus\Om$. To this end, observe that the function $v$ previously introduced is defined and analytic all over $\R^n\setminus\{0\}$. Thus, since $u$ is analytic as well in $M$ and coincides with $v$ on an open subset of $M$, they must coincide everywhere. In particular, $u$ is rotationally symmetric in $\R^n\setminus\Om$ and $\pa\Om=\pa B_{r_0}$, 
with $r_0=r_c\,c^{1/(n-2)}$.  
\end{proof}

Now, we proceed with the proof of Theorem \ref{thm:Lore_Virgi_delta},
which pertains to the problem in the interior punctured domain with the integral condition (\ref{eq:cond_Lore_Virgi_delta}).

\begin{proof}[Proof of Theorem \ref{thm:Lore_Virgi_delta} (after Proposition \ref{prop:int_part_2})]
Let $u$ be the unique solution of problem (\ref{eq:pb_delta}).
In order to apply Proposition \ref{prop:int_part_2} to the function $f=\log u$, we need to check that the condition 
\begin{equation}
\label{eq:ancora_f}
\int\limits_{\{f=\log c\}}\!\!\!|\na f|_g^3\,{\rm d}\sigma_g
\,
=
\,\,
\limsup_{t\to\infty}\!\!\!\int\limits_{\{f=\log t\}}\!\!\!|\na f|_g^3\,{\rm d}\sigma_g
\end{equation}
is satisfied. In terms of $u$, the above condition is equivalent to 
\begin{equation}
\label{eq:cond_su_c}
\int\limits_{\{u=c\}}\frac{|{\rm D}u|^2}{u^{2\frac{n-1}{n-2}}}|\DDD u|\,{\rm d}\sigma
\,
=
\,\,
\limsup_{t\to+\infty}\int\limits_{\{u=t\}}\frac{|{\rm D}u|^2}{u^{2\frac{n-1}{n-2}}}|\DDD u|\,{\rm d}\sigma.
\end{equation}
Note that the limsup on the right-hand side is indeed a limit and is independent of the Dirichlet data. 
To see this, let us observe that, in view of Lemma \ref{lem:delta_behaves}, the solution $u$ can be written as
\[
u(x)=\frac{d|\pa\Om|}{(n-2)|\Sph^{n-1}|}|x|^{2-n}+v(x), 
\]
for some function $v$ harmonic in $\Om$.
In particular, from this representation one can deduce that 
the identity
\begin{equation*}
\int\limits_{\{u=t\}}|\DDD u|\rmd\sigma
=
\int\limits_{\pa\Om}|\DDD u|\rmd\sigma
=
d\,|\pa\Om|
\end{equation*}
holds for every $t\geq c$, and that 
\begin{equation*}
\lim_{x\to0}
\frac{|\DDD u|}{u^{\frac{n-1}{n-2}}}
=
(n-2)^{\frac{n-1}{n-2}}\left(\frac{|\Sph^{n-1}|}{d\,|\pa\Om|}\right)^{\frac1{n-2}}.
\end{equation*}
From the above identity and limit we get
\begin{equation}
\label{eq:calc_limit}
\lim_{t\to+\infty}\int\limits_{\{u=t\}}\frac{|\DDD u|^2}{u^{2\frac{n-1}{n-2}}}|\DDD u|\,{\rm d}\sigma
=
(n-2)^{\frac{2(n-1)}{n-2}}\left(\frac{|\Sph^{n-1}|}{d\,|\pa\Om|}\right)^{\frac2{n-2}}
\int\limits_{\pa\Om}|\DDD u|\,{\rm d}\sigma,
\end{equation}
which is a quantity that does not depend on the constant value of $u$ on $\Om$. Therefore, we have that condition (\ref{eq:cond_su_c})
is fulfilled if $c=c_1$, with 
\begin{equation*}
c_1:=
\left(
\frac{\displaystyle\int_{\pa\Om}|\DDD u|^3\,\rmd\sigma}
{\displaystyle\lim_{t\to+\infty}\int\limits_{\{u=t\}}\frac{|\DDD u|^2}{u^{2\frac{n-1}{n-2}}}|\DDD u|\,{\rm d}\sigma}
\right)^{\frac{n-2}{2(n-1)}}\!\!\!\!.
\end{equation*}
To stress the fact that $c_1$ only depends on $n$, $\Om$ and $d$, 
let us denote by $u_0$ the solution to problem (\ref{eq:pb_delta}) with homogeneous Dirichlet boundary conditions and use (\ref{eq:calc_limit}) to write 
\begin{equation}
\label{def:c_n_1}
c_1=c_1(n,\Om,d)=
\frac1{(n-2)}\left(\frac{|\pa\Om|}{|\Sph^{n-1}|}\right)^{\frac1{n-1}}
\frac{\displaystyle\left(\fint_{\pa\Om}|\DDD u_0|^3\,\rmd\sigma\right)^{\frac{n-2}{2(n-1)}}}
{\displaystyle\left(\fint_{\pa\Om}|\DDD u_0|\,\rmd\sigma\right)^{\frac{n-4}{2(n-1)}}}\,\cdot
\end{equation}
By this discussion, we have that, up to an additive constant, it is always possible to suppose that $u=c_1$ on $\pa\Om$, so that  
(\ref{eq:cond_su_c}) and in turn (\ref{eq:ancora_f}) are satisfied. 
Observe that this assumption is not restrictive, because
adding a constant does not affect condition (\ref{eq:cond_Lore_Virgi_delta}).
Moreover, it is a matter of computation to check that, in this normalized setting, condition (\ref{eq:cond_Lore_Virgi_delta}) can be rewritten as
\begin{equation*}
\int\limits_{\{u=c_1\}}|\DDD u|^2
\left[\frac{\HHH}{n-1}-\frac{|\DDD u|}{(n-2)\,u}\right]\rmd\sigma\geq0
\end{equation*}
and in turn, by (\ref{eq:ipotesi_riscritta}), as
\[
\int\limits_{\{f=\log c_1\}}|\na f|_g^2\,\HHH_g\,\rmd\sigma_g\geq0.
\]
Therefore, both conditions of hypothesis (\ref{eq:cond_curv_media_2})
in Proposition \ref{prop:int_part_2} are satisfied.
Hence, by analogy with the proof of Theorem \ref{thm:Lore_Virgi}, 
we can apply Proposition \ref{prop:int_part_2} and Theorem \ref{prop:LCF_sfera} to deduce that  
the Riemannian manifold $\big(\pa\Om,{g_{\R^n}}_{|\pa\Om}\big)$, 
is isometric to a constant curvature sphere and that
$u$ is rotationally symmetric.  
\end{proof}

We now pass to consider the problem in the exterior domain (\ref{eq:pb}) with pointwise over determining conditions. 

\begin{proof}[Proof of Theorem \ref{thm:nabla_const} (after Proposition \ref{prop:const_hyp})]
Let $u$ be a solution to problem (\ref{eq:pb})  and let $g$ be the metric defined in (\ref{eq:tilde_g}). Setting $f=\log u$, we have that the system (\ref{eq:pb_reform}) is satisfied in $M=\R^n\setminus\Om$. According to (\ref{eq:formula_H_H_g}) and (\ref{eq:P_function}), the hypothesis (\ref{eq:hyp_nabla_const}) is equivalent to
\begin{equation*}
|\na f|_g=\frac d{c^{\frac{n-1}{n-2}}}\ \ \mbox{ on $\ \{f=\log c\}$}\qquad\mbox{ and }\qquad\inf_{\{f=\log c\}}\HHH_g\leq0.
\end{equation*}
Thus, Proposition \ref{prop:const_hyp} yields $\na^2f\equiv0$ in $\{f\leq\log c\}$.
Finally, arguing as in Theorem \ref{thm:Lore_Virgi}, we can deduce, via Theorem \ref{prop:LCF_sfera}, that 
$\big(\pa\Om,{g_{\R^n}}_{|\pa\Om}\big)$ is isometric to a spherical space form and that $u$ is rotationally symmetric.
\end{proof}

We continue with the proof of Theorem
\ref{thm:nabla_const_delta},
which deals with the problem in the interior punctured domain with the overdetermining pointwise condition (\ref{eq:hyp_nabla_const_delta}).

\begin{proof}[Proof of Theorem \ref{thm:nabla_const_delta} (after Proposition \ref{prop:const_hyp_2})]

Let $u$ be the unique solution of problem (\ref{eq:pb_delta})
satisfying $|\DDD u|=d$ on $\pa\Om$.
In order to apply Proposition \ref{prop:const_hyp_2} to the function $f=\log u$, we need to check that the condition 
\begin{equation*}
|\na f|_g
\,
\equiv\,\,
\max_M|\na f|_g\ \ \mbox{ on $\ \pa M$}
\end{equation*}
is satisfied. 
To do this, it is sufficient to prove that
\begin{equation}
\label{eq:ancora_nabla_f}
|\na f|_g^2
\equiv
\limsup_{x\to0}|\na f|_g^2\qquad\mbox{on }\ \pa M.
\end{equation}
Indeed, from equation (\ref{eq:Bochner_g}) we have 
\begin{equation*}
\Delta_g|\na f|_g^2+\big\langle\na|\na f|^2_g\,\big|\,\na f\big\rangle_g\geq0,
\end{equation*}
so that, by the Maximum Principle and by (\ref{eq:ancora_nabla_f}), we have $\max_{\pa M}|\na f|_g^2=\max_M|\na f|_g^2$.
In terms of $u$, condition (\ref{eq:ancora_nabla_f}) is equivalent to 
\begin{equation}
\label{eq:cond_su_c_2}
\frac{|\DDD u|^2}{u^{2\frac{n-1}{n-2}}}
\,\equiv\,\,
\limsup_{x\to0}\frac{|\DDD u|^2}{u^{2\frac{n-1}{n-2}}}
\qquad\mbox{on }\ \pa\Om.
\end{equation}
As seen in the proof of Theorem \ref{thm:Lore_Virgi_delta},
Lemma \ref{lem:delta_behaves} implies that the $\limsup$ in this condition is indeed a limit and coincides with 
$(n-2)^{2\frac{n-1}{n-2}}(|\Sph^{n-1}|/d\,|\pa\Om|)^{\frac2{n-2}}$.
Therefore, the above condition holds true if $u=c_2$ on $\pa\Om$, with 
\begin{equation}
\label{def:c_n_2}
c_2=c_2(n,\Om,d):=\frac d{n-2}\left(\frac{|\pa\Om|}{|\Sph^{n-1}|}\right)^{\frac1{n-1}}\!\!\!\!.
\end{equation}
Note that, up to an additive constant, it is always possible to suppose that $u=c_2$ on $\pa\Om$, 
because adding a constant does not affect hypothesis (\ref{eq:hyp_nabla_const_delta}). 
Moreover, by this choice of the Dirichlet boundary data, simple computations show that the second condition in hypothesis (\ref{eq:hyp_nabla_const_delta}) is equivalent to
\[
0\,\leq\,
\sup_{\pa\Om}\frac{\HHH}{n-1}-
\frac{|\DDD u|}{(n-2)\,u}
\,=\,
\sup_{\pa\Om}\bigg[\frac{\HHH}{n-1}-\frac{|\DDD u|}{(n-2)\,u}\bigg].
\]
In turn, using (\ref{eq:formula_H_H_g}), this is equivalent to
$\sup_{\pa M}\HHH_g\geq0$. 
Therefore, we are in the position to use Proposition \ref{prop:const_hyp_2}, which yields $\na^2f\equiv0$ in $M$.
Finally, arguing as in Theorem \ref{thm:Lore_Virgi}, we can deduce, via Theorem \ref{prop:LCF_sfera}, that 
$\big(\pa\Om,{g_{\R^n}}_{|\pa\Om}\big)$ is isometric to a spherical space form and that $u$ is rotationally symmetric.
\end{proof}

We now present the proof of the result for annular domains.

\begin{proof}[Proof of Theorem \ref{thm:2bordi}]
Let $u$ be the solution to either problem (\ref{eq:pb}) or problem (\ref{eq:pb_delta}). 
As usual, we set $f=\log u$.
Observe that, via identity (\ref{eq:ipotesi_riscritta}),
the (pointwise) conditions satisfied by $u$ on the level sets $\{u=a\}$
and $\{u=b\}$ imply that 
\begin{equation*}
\int\limits_{\{f=\log a\}}\!\!|\na f|_g^2\,\HHH_g\,{\rm d}\sigma_g\,\geq\,0
\qquad\mbox{and}\qquad
\int\limits_{\{f=\log b\}}\!\!|\na f|_g^2\,\HHH_g\,{\rm d}\sigma_g\,\leq\,0,
\end{equation*}
where $0<a<b$.
Therefore, a straightforward application of Corollary \ref{cor:int_pimo} with $\varphi(f)=f$ gives that 
$\na^2f\equiv0$ in 
$\{\log a\leq f\leq \log b\}$. 
Hence, $|\na f|_g$ is a positive constant in the same region, 
so that all the level sets $\{u=t\}$, for $a\leq t\leq b$, 
are regular, diffeomorphic to each other and in turn
connected by hypothesis.
The conclusion then follows by the same arguments employed in the previous proofs.
The only difference amounts to observing that a local version of 
Theorem~\ref{prop:LCF_sfera} holds, with a finite interval in place of the half straight line $[0,+\infty)$.
\end{proof}

We conclude this section with the proof of 
Corollary \ref{cor:cond_glob}. 

\begin{proof}[Proof of Corollary \ref{cor:cond_glob}]
Let us first note that, by the coarea formula and 
by the first identity in (\ref{eq:formula_curvature}), 
we have
\begin{equation}
\label{eq:id_cor_1.2}
\int\limits_{\R^n\setminus\Om}|\DDD u|^3
\bigg(\,\frac{\DDD^2u(\DDD u,\DDD u)}{(n-1)|\DDD u|^3}
-\frac{|\DDD u|}{(n-2)u}\,\bigg)\rmd\mu
=
\int\limits_0^1\!\int\limits_{\{u=c\}}\!|\DDD u|^2
\bigg(\frac{\HHH}{n-1}
-\frac{|\DDD u|}{(n-2)u}\bigg)\rmd\sigma\,\rmd c,
\end{equation}
where we observe that the slice integral
\[
\int\limits_{\{u=c\}}\!|\DDD u|^2
\bigg(\frac{\HHH}{n-1}
-\frac{|\DDD u|}{(n-2)u}\bigg)\rmd\sigma
\]
is defined for a.e.\ $0<c<1$, in view of Sard's Lemma.
Therefore, if the left-hand in \eqref{eq:id_cor_1.2} is nonpositive, 
we obtain that condition (\ref{eq:cond_Lore_Virgi}) of 
Theorem \ref{thm:Lore_Virgi} is satisfied for some
regular value $c\in(0,1]$, 
so that the conclusion holds.
Now, observing that the quantities $\DDD^2u(\DDD u,\DDD u)$
and $|\DDD u|^4/u$ are integrable in $\R^n\setminus\overline{\Om}$
thanks to Lemmata \ref{lem:u_behaves} and \ref{lem:u_deriv_behaves},
an application of the Divergence Theorem gives
\begin{equation*}
\int\limits_{\R^n\setminus\Om}
\!\DDD^2u(\DDD u,\DDD u)\,\rmd\mu
=
\frac12\int\limits_{\pa\Om}|\DDD u|^3\,\rmd\sigma.
\end{equation*}
Note that we have also use the fact that 
$\Delta u=0$ in $\R^n\setminus\overline{\Om}$.
In view of the previous identity, the non positivity
of the left hand side in (\ref{eq:id_cor_1.2}) translates into
\begin{equation*}
\int\limits_{\{u=1\}}\frac{|\DDD u|^3}u\rmd\sigma
\,=\,
\int\limits_{\pa\Om}|\DDD u|^3\rmd\sigma
\,\leq\,
2\,
\Big(\frac{n-1}{n-2}\Big)
\!\int\limits_{\R^n\setminus\Om}\!\frac{|\DDD u|^4}u\rmd\mu
\,=\,
2\,
\Big(\frac{n-1}{n-2}\Big)
\int\limits_0^1\!\int\limits_{\{u=c\}}\!
\frac{|\DDD u|^3}u\,\rmd\sigma\,\rmd c,
\end{equation*}
where in the second equality we have again used the coarea formula.
Recalling that
\begin{equation*}
\Phi(c)=\int\limits_{\{u=c\}}\!\frac{|\DDD u|^3}u\,\rmd\sigma, 
\end{equation*}
the previous inequality coincides with the condition 
in (\ref{eq:cond_Phi}) and the thesis follows.
\end{proof}


\renewcommand{\theequation}{A-\arabic{equation}}
\renewcommand{\thesection}{A}
\setcounter{equation}{0}  
\setcounter{theorem}{0}  
\section*{Appendix: asymptotic behaviour of the solutions}  

In this section we investigate the asymptotic behaviour of solutions to problem (\ref{eq:pb}) and (\ref{eq:pb_delta}). The estimates presented here are well known, since they can be deduced by standard arguments in the theory of elliptic equations. However, for sake of completeness, we provide the proofs.
We start with the following upper and lower bounds for solutions to problem (\ref{eq:pb}).

\begin{lemma}
\label{lem:u_behaves}
If $u$ is a solution to problem (\ref{eq:pb}), then there exists $r_0>0$
and two positive constants $C_1<C_2$ such that
\begin{equation}
\label{eq:u_cresc}
C_1|x|^{2-n}\leq u(x)\leq C_2|x|^{2-n},\qquad\mbox{for every  }\quad|x|\geq r_0.
\end{equation}
\end{lemma}

\begin{proof}
Note first that $0<u<1$ on $\R^n\setminus\overline\Om$, in view of the Strong Maximum Principle. 
Let then $r_0>0$ be such that
$\Om\subset\subset B_{r_0}$, and set
\[
C_1:=r_0^{n-2}\min_{\pa B_{r_0}}u,\qquad\qquad C_2:=r_0^{n-2}\max_{\pa B_{r_0}}u.
\]
We define for $r>0$
\[
u^-(x):=C_1|x|^{2-n},\qquad\qquad
u^+(x):=C_2|x|^{2-n}.
\]
Observe that $(u^--u)\leq0$ on $\pa B_{r_0}$ and that $(u^--u)\to0$ as $|x|\to\infty$. Since $\Delta(u^--u)=0$ in $\R^n\setminus B_{r_0}$,  a simple argument based on the Strong Maximum Principle shows that $(u^--u)\leq0$ in $\R^n\setminus B_{r_0}$. Similarly, we have that $(u^+-u)\geq0$ on $\R^n\setminus B_{r_0}$.
\end{proof}

Building on the previous lemma, one can derive upper bounds for the
first and the second derivatives of $u$.

\begin{lemma}
\label{lem:u_deriv_behaves}
If $u$ is a solution to problem (\ref{eq:pb}), then there exists $r_0>0$
and two positive constants $C_3$ and $C_4$ such that
\begin{equation}
\label{eq:u_deriv_cresc}
|\DDD u(x)|\leq C_3|x|^{1-n}\quad\mbox{ and }\quad 
|\DDD^2u(x)|\leq C_4|x|^{-n},
\qquad\mbox{for every  }\quad|x|\geq2r_0.
\end{equation}
\end{lemma}

\begin{proof}
First note that, since $u$ is harmonic, we have in particular that the partial derivative $\pa_{\alpha}u$ of $u$ is harmonic as well,
for every $\alpha=1,...,,n$. Hence, by the Mean Value Property, we have that for every $x\in\R^n\setminus\overline\Om$ and every $\rho>0$ such that $B_{\rho}(x)\subset\subset\R^n\setminus\overline\Om$,  
\[
\pa_{\alpha}u(x)=
\frac1{\omega_n\rho^n}\int\limits_{B_{\rho}(x)}\pa_{\alpha}u\,\rmd\mu=
\frac1{\omega_n\rho^n}\int\limits_{\pa B_{\rho}(x)}u\,\langle\nu|\pa_{\alpha}\rangle\,\rmd\sigma,
\]
where $\nu$ is the outer unit normal on $\pa B_{\rho}(x)$.
Therefore, by the positivity of $u$, we get
\begin{equation}
\label{eq:from_pos_u}
|\pa_{\alpha}u(x)|\leq\frac1{\omega_n\rho^n}\int\limits_{\pa B_{\rho}(x)}u\,\rmd\sigma=\frac n{\rho}u(x),
\end{equation}
where in the equality we have used the fact that 
$\int_{\pa B_{\rho}(x)}u\,\rmd\sigma=n\omega_n\rho^{n-1}u(x)$.
In particular, fixing $r_0>0$ such that $\Om\subset\subset B_{r_0}$,
we obtain 
\[
|\pa_{\alpha}u(x)|\leq\frac n{|x|-r_0}u(x)\leq
\frac{nC_2|x|^{2-n}}{|x|-r_0}
\leq2nC_2|x|^{1-n},
\]
for every $|x|\geq2r_0$,
where in the second inequality we have used Lemma \ref{lem:u_behaves}.
Hence, there exists  $C_3>0$ such that the first of (\ref{eq:u_deriv_cresc}) holds true.

To obtain the second estimate in (\ref{eq:u_deriv_cresc}), 
we proceed as before and  we get  
\[
\pa_{\alpha}\pa_{\beta}u(x)=
\frac1{\omega_n\rho^n}\int\limits_{\pa B_{\rho}(x)}\pa_{\alpha}u\,\langle\nu|\pa_{\beta}\rangle\,\rmd\sigma,
\]
for every $x\in\R^n\setminus\overline\Om$ and every $\rho>0$ such that $B_{\rho}(x)\subset\subset\R^n\setminus\overline\Om$.
Thus,
\begin{equation*}
|\pa_{\alpha}\pa_{\beta}u(x)|\leq
\frac{n}{\omega_n\rho^{n+1}}\int\limits_{\pa B_{\rho}(x)}u\,\rmd\sigma=\frac{n^2}{\rho^2}u(x),
\end{equation*}
where in the inequality we have used (\ref{eq:from_pos_u}), whereas the equality follows from the Mean Value Property.
In particular, fixing $r_0>0$ as before, we deduce from the above estimate that
\[
|\pa_{\alpha}\pa_{\beta}u(x)|\leq
\frac{n^2}{(|x|-r_0)^2}u(x)\leq\frac{n^2C_2|x|^{2-n}}{(|x|-r_0)^2}
\leq4n^2C_2|x|^{-n},
\]
for every $|x|\geq2r_0$.
Note that the second inequality follows from Lemma \ref{lem:u_behaves}. The last estimate implies that
the second of (\ref{eq:u_deriv_cresc}) holds true for some $C_4>0$.
\end{proof}

\begin{remark}
\label{rm:diffeo}
\upshape
Let $u$ be the solution to problem (\ref{eq:pb}). 
Since $\pa\Om$ is supposed to be smooth,
according to \cite[Lemma 3.4]{Gil_Tru_book} we have that every level set 
$\{u=c\}$, for $c\leq1$ sufficiently close to $1$, is regular.
Also, thanks to Lemmata \ref{lem:u_behaves} and \ref{lem:u_deriv_behaves}, 
for every $\ep>0$ sufficiently small the level set $\{u=\ep\}$ 
is regular and diffeomorphic to a $(n-1)$-dimensional sphere.
Note that it may happen that $\DDD u$ vanishes at some point 
of $\R^n\setminus\overline\Om$,
so that nonregular level sets of $u$ may exist.
However, observe that the set $\big\{c\in(0,1):\{u=c\}\mbox{ is regular}\big\}$
has full Lebesgue measure, in view of Sard's Lemma.
Finally, recall that if the sublevel $\{u< c\}$ satisfies the 
interior sphere condition, then, by \cite[Lemma 3.4]{Gil_Tru_book},
one can conclude that $\{u=c\}$ is regular. 
\end{remark}

We now turn the attention to the solutions of problem (\ref{eq:pb_delta}).
The following lemma describes their asymptotic behaviour.

\begin{lemma}
\label{lem:delta_behaves}
Let $u$ be the solution to the problem
\begin{equation}
\label{pb:delta_appendice}
\left\{
\begin{array}{rcll}
\displaystyle
-\Delta u\!\!\!\!&=&\!\!\!\!d|\pa\Om|\delta_0 & {\rm in }\quad\Om,\\
\displaystyle
  u\!\!\!\!&=&\!\!\!\!c  &{\rm on }\ \ \pa\Om,\\
\end{array}
\right. 
\end{equation}
where $\Om$ is a bounded domain with smooth boundary which contains the origin and $d>0$, $c\!\in\!\R$ are constants.
Then, there exists a unique harmonic function $v$ in $\Om$ such that
\begin{equation*}
u(x)=\frac{d|\pa\Om|}{(n-2)|\Sph^{n-1}|}|x|^{2-n}+v(x),\qquad\quad\mbox{for every $x\in\Om\setminus\{0\}$.}
\end{equation*}
\end{lemma}

\begin{proof}
Let us set
$a_n:=1/(n-2)|\Sph^{n-1}|$.
It is well-known that $-\Delta(a_n|x|^{2-n})=\delta_0$, so that
$-\Delta(u-a_nd|\pa\Om||x|^{2-n})=0$ in $\Om$, in the sense of distributions.
On the other hand, by the classical regularity theory, a distributional solution of this equation is actually smooth in $\Om$. 
By the maximum principle this solution coincides with the unique solution $v$ of
\begin{equation*}
\left\{
\begin{array}{rcll}
\displaystyle
-\Delta v\!\!\!\!&=&\!\!\!\!0 & {\rm in }\ \ \Om,\\
\displaystyle
  v\!\!\!\!&=&\!\!\!\!c-a_nd|\pa\Om||x|^{2-n}  &{\rm on }\ \ \pa\Om.
\end{array}
\right.
\end{equation*}
This completes the proof of the lemma.
\end{proof}


\subsection*{Acknowledgements}
\emph{
V.~A. has received funding from the 
European Research Council under the 
European Union's Seventh Framework Programme (FP7/2007-2013) / ERC grant agreement ${\rm n^o}$ 291053.
L.~M. has been partially supported by the Italian project FIRB 2012 ``Geometria Differenziale e Teoria Geometrica delle Funzioni'' as well as by the SNS project ``Geometric flows and related topics''.
The authors, who are members of the Gruppo Nazionale per
l'Analisi Matematica, la Probabilit\`a e le loro Applicazioni (GNAMPA)
of the Istituto Nazionale di Alta Matematica (INdAM),
would like to thank G. Buttazzo, R. Magnanini, M. Novaga, and 
B. Velichkov for interesting discussions during the preparation 
of this paper.
}


\bibliographystyle{plain}

\end{document}